\documentclass{article}

\usepackage{amsmath} 
\usepackage{amssymb}  
\usepackage{enumerate}
\usepackage{fullpage}
\usepackage{xcolor}

\usepackage[bookmarks]{hyperref}

\newcommand{\N}{\mathbb{N}}
\newcommand{\RR}{\mathbb{R}}                                  

\newcommand{\WW}{\mathcal{W}}                                  
\newcommand{\UU}{\mathcal{U}}
\newcommand{\MM}{\mathcal{M}}

\newcommand{\bd}{\mathrm{bd}\,}
\newcommand{\intt}{\mathrm {int}\,}
\newcommand{\diag}{\mathrm{diag}\,}
\newcommand{\conv}{\mathrm{conv}\,}
\newcommand{\midset}{\;|\;}
\def\Expect{\mathop{\bf E{}}}
\def\Prob{{ \rm I \!\! P}}
\def\dsp{\displaystyle}

\newcommand{\qed}{$\Box$}

\newtheorem{definition}{Definition}[section]

\newtheorem{theorem}{Theorem}[section]

\newtheorem{proposition}{Proposition}[section]

\newtheorem{lemma}{Lemma}[section]
\newtheorem{corollary}{Corollary}[section]
\newtheorem{remark}{Remark}[section]

\newenvironment{proof}[1][Proof]{\begin{trivlist}
\item[\hskip \labelsep {\bfseries #1}]}
{\end{trivlist}}

\title{Stability Criteria for SIS Epidemiological Models under Switching Policies}

\author{%
Mustapha Ait Rami\thanks{Dept. Ingenieria de Sistemas y Automatica, Universidad de Valladolid, Valladolid, Spain ({\tt aitrami@autom.uva.es})}
\and
Vahid S. Bokharaie\thanks{Departement Werktuigkunde, KU Leuven, Leuven, Belgium 
({\tt vahid.bokharaie@kuleuven.be})}
\and
Oliver Mason\thanks{Hamilton Institute, National University of Ireland Maynooth,Maynooth, Ireland ({\tt oliver.mason@nuim.ie})}
\and
Fabian R. Wirth\thanks{IBM Research Ireland, Damastown Industrial Estate, Mulhuddart, Dublin 15, Ireland ({\tt fabwirth@ie.ibm.com})}
}

\begin{document}

\maketitle

\begin{abstract}      
    We study the spread of disease in an SIS model for a structured
    population. The model considered is a time-varying, switched model, in
    which the parameters of the SIS model are subject to abrupt change.
    We show that the joint spectral radius can be used as a threshold
    parameter for this model in the spirit of the basic reproduction
    number for time-invariant models.  We also present conditions for
    persistence and the existence of periodic orbits for the switched
    model and results for a stochastic switched model.
\end{abstract}



\pagestyle{myheadings}
\thispagestyle{plain}
\markboth{M. AIT RAMI, V. BOKHARAIE, O. MASON,  AND, F. WIRTH}{STABILITY FOR SIS EPIDEMIOLOGICAL MODELS}

\parindent0pt
\section{Introduction} \label{sec:intro} In this paper a number of
stability results are derived for switched compartmental epidemiological
models of SIS (susceptible-infectious-susceptible) type. Such models are
related to structured populations to which an infection graph is
associated. The switched system aspect models changes is
parameters, e.g. in infection or recovery rates. We derive uniform
stability results for the disease free equilibrium, as well as instability
results, study the existence of nontrivial periodic solutions and present
some results on Markovian switching.  Mathematically, the results rely on
a combination of the theory of positive and monotone systems, Lyapunov
theory for switched systems, tools from degree theory and Markov systems.

{Mathematical epidemiology is concerned with the construction and analysis of mathematical models for disease propagation in a network of individuals; similar models can also be applied to study the spread of computer viruses.   The role of mathematical modelling is particularly important in epidemiology as experimentation in this field is both impractical and unethical.  As mentioned in \cite{Bai57}: ``we need to develop models that will assist the decision-making process by helping to evaluate the consequences of choosing one of the alternative strategies available. Thus, mathematical models of the dynamics of a communicable disease can have a direct bearing on the choice of an immunisation programme, the optimal allocation of scarce resources, or the best combination of control or education technologies''.  To further underline the importance of the topic, it should be noted that in spite of the advances in vaccination and prevention of disease transmission in the past few decades, infectious and parasitic diseases are the second leading cause of death worldwide (after cardiovascular diseases) \cite[Figure 4]{WHO08}. They are the leading cause of death in low-income countries \cite[Table 2]{WHO08} and in children aged under five years \cite[Figure 5]{WHO08}.

Epidemiological models based on ordinary or partial differential equations typically divide the population into different epidemiological classes or compartments.  For instance, such compartments may represent:  \emph{susceptible} (S) individuals who are healthy but not immune to the disease; \emph{infective} individuals (I), who are already infected and can transfer the disease to a susceptible; or \emph{recovered} individuals (R) who have immunity to the disease.  Some other more complex models also include compartments for exposed people who are infected but not yet infective and for infants with temporary inherited immunity.  A model is typically represented by the initials of the epidemiological classes it incorporates.

In this manuscript, we are concerned with SIS models, in which all individual are considered to be either susceptible or infective.  When susceptibles are in sufficient contact with infectives, they become infectives themselves. When infectives are restored to health, they re-join the susceptible compartment.
SIS models have been used to model diseases that do not confer immunity on the survivors. Tuberculosis and gonorrhoea are two example diseases which are mathematically described using SIS models \cite{CY73,HY84,New03}. Computer viruses also fall into this category;  they can be `cured' by antivirus software, but without  a permanent  virus-checking  program, the computer is not protected against the subsequent attacks by the  same virus.

Our work in this paper seeks to build on the results presented in \cite{FIST07}, where an SIS model, with constant parameters, that can describe heterogeneous populations was considered.  The authors of \cite{FIST07} presented conditions for the stability of the disease-free equilibrium and for the existence and stability of a unique endemic equilibrium.  These conditions were described in terms of the spectral radius of a matrix naturally associated with the system and are in the spirit of the more general model class studied in \cite{DW02}.  It has been recognized for some time that time-varying parameters play an important role in the dynamics of disease propagation; in particular, several authors have considered the impact of seasonal effects by analysing periodic SIS models \cite{PerSeasonal2012, PerPatchy2011}.

A major concern in the current paper is to study a switched version of the model in \cite{FIST07}.  This allows us to consider the effect of sudden changes in the parameters of the model.  Such changes can arise for a variety of reasons.  For instance, public health authorities may implement a rapid vaccination programme or close schools or public transport systems.  Such policies generate abrupt changes in key model parameters.  It is worth pointing out that with modern communication systems it is possible for large groups of individuals to collectively alter their behaviour rapidly through online rumour spreading for instance.  The sudden variation in model parameters that this would give rise to need not be periodic and provides some practical motivation for considering a switched model.

We show that the results of \cite{FIST07} may be generalized in a far
reaching manner. Our first main result is that for switched systems
uniform stability of the linearization implies uniform global asymptotic
stability. The result may be formulated using the assumption of
irreducibility, in which case they become direct extensions of
\cite{FIST07}.  In addition, we present conditions for the existence of
instability and the existence of endemic periodic solutions even if the
constituent systems of the switched system are all stable. In this context
we will several times make use of the theory of monotone systems studied
in \cite{Smi95,Chue02}. One particular contribution of this paper is to
analyse the stability of the epidemic model in \cite{FIST07} when its
parameters are subject to random Markovian switching. General dynamical
systems with random Markovian switching were introduced and studied in
\cite{KaK:60} and \cite{KrL:61}. In the literature, such systems are
frequently called {\em jump systems} and are connected to a wide range of
applications; for more details we refer to~\cite{Mar:90}.  For the problem
of uniform global asymptotic stability it is shown in \cite{Proell13} that
the techniques of the present paper may be used to also treat SIR and SIRS
models.

The layout of the paper is as follows: in Section \ref{sec:bkground}, we
present the preliminary results and basic properties which are used in the
remainder of this manuscript.  The basic SIS model of \cite{FIST07} and
the switched system framework we study in this manuscript are described in
Section \ref{sec:problem}.  We describe results on Positive Switched
Linear Systems and extremal norms in Section~\ref{sec:PosLinSys}.
Building on this, results characterizing global uniform asymptotic
stability of the disease-free equilibrium of the switched SIS model in
terms of the joint spectral radius are given in Section \ref{sec:DFE}.  In
Section \ref{sec:Per}, we consider the case where the Disease Free
Equilibrium (DFE) is a globally asymptotically stable equilibrium of all
the subsystems and introduce a switching signal that gives rise to endemic
behaviour.  Section \ref{sec:markov} examines the stability of the DFE of
the switched SIS model when the switching signal is a Markov process. This
case has already been studied in \cite{BacaKhal12,GrayGree12}, where the
interpretation of the basic reproduction number is studied for the case of
populations with age structure or with no structure. The case where each
subsystem of the switched SIS model has an endemic equilibrium is
considered in Section \ref{sec:Stabilisation} and conditions for the
existence of a switching signal for which the DFE is a globally
asymptotically stable equilibrium of the switched SIS model are
described. In Section \ref{sec:conclusions}, we present our conclusions.
}
\section{Preliminaries}\label{sec:bkground}
{
{
  Throughout the paper, $\RR$ and $\RR^n$ denote the field of real numbers
  and the vector space of $n$ dimensional column vectors with real
  entries, respectively.
  For $x \in \RR^n$ and $i = 1, \ldots , n$, $x_i$ denotes the $i$th
  coordinate of $x$. Similarly, $\RR^{n \times n}$ denotes the space of $n
  \times n$ matrices with real entries and for $A \in \RR^{n \times n}$,
  $a_{ij}$ denotes the $(i, j)$th entry of $A$.  The positive orthant
  in $\RR^n$ is $\RR^n_+ := \{x \in \RR^n: x_i \geq 0, 1 \leq i \leq n\}.$
  The interior of $\RR^n_+$ is denoted by $\intt(\RR^n_+)$ and its
  boundary by $\bd(\RR^n_+):= \RR^n_+ \setminus \intt(\RR^n_+)$.
  For vectors $x, y \in \RR^n$, we write: $x \geq y$ if $x_i \geq y_i$ for
  $1 \leq i \leq n$; $x > y$ if $x \geq y$ and $x \neq y$; $x \gg y$ if
  $x_i > y_i, 1 \leq i \leq n$. The absolute value $|x|$ of a vector $x\in
  \RR^n$ is defined by $|x|_i := |x_i|, i=1,\ldots,n$.

A matrix $A\in \RR^{n \times n}_+ $ is called \emph{irreducible} if for
every nonempty proper subset $K$ of $N:=\{1, \cdots, n\}$, there exists an
$i\in K$, $j \in N \setminus K$ such that $a_{ij}\neq 0$. When $A$ is not
irreducible, it is \emph{reducible}.
Also, for $x\in \RR^n$, $\diag(x)$ is the $n \times n$
diagonal matrix in which $d_{ii}=x_i$.  \\
For $A \in \RR^{n \times n}$, we denote the \textit{spectrum} of $A$
by $\sigma(A)$ and the spectral radius of $A$ by $\rho(A)$.  The
notation $\mu(A)$ denotes the \textit{spectral abscissa} of $A$
which is defined as follows:
\begin{equation}\nonumber
\mu(A) := \max \{ \mathrm{Re} (\lambda) : \lambda \in \sigma(A) \}.
\end{equation}
A matrix $A\in \RR^{n \times n}$ is called {\em Hurwitz}, if $\mu(A)<0$.

  It will be useful to study norms, which are adapted to the nonnegative
  setting.  A norm on $\RR^n$ is called {\em monotone} if $|x|\geq |y|$
  implies $\|x\|\geq \|y\|$. This is equivalent to the requirement that
  $\|x\| = \| \, |x|\, \|$ for all $x\in \RR^n$, see
  \cite[Theorem~2]{bauer1961absolute}, \cite[Theorem~5.5.10]{HornJohn}.
  Norms with the latter property are called {\em absolute} and we will use
  this name throughout the remainder of the paper.

  The dual norm $\|\cdot\|^*$ of a norm $\|\cdot\|$ on $\RR^n$ is defined by
  \begin{equation}
      \label{eq:defdualnorm}
      \|y\|^* := \max \{ \langle x,y \rangle \midset \|x\|\leq 1 \}\,,\quad y \in \RR^n \,.
  \end{equation}
  It is known that a norm is absolute if and only if its dual norm is,
  \cite[Theorem~1]{bauer1961absolute}. Given $\|\cdot\|$ and its dual
  $\|\cdot\|^*$, a pair of vectors $(x,y)\in \RR^n\times \RR^n $ is called
  a {\em dual pair}, if
  \begin{equation}
      \label{eq:dualpair}
      \langle x,y \rangle = \|x\|\, \|y\|^*\,.
  \end{equation}
  A vector $y$ is called {\em dual} to $x\in \RR^n$, if $(x,y)$ is a dual
  pair. Note that the order is important: in a dual pair the first vector
  is evaluated with $\|\cdot\|$ and the second with its dual. The relation
  is not symmetric, in general.

 The following observation on the
  properties of dual vectors of absolute norms will be useful.
  \begin{lemma}
      \label{lem:dualabs}
      Let $\| \cdot\|$ be an absolute norm on $\RR^n$ and $x\in \RR^n_+$. If
      $y$ is dual to $x$ and $x_i>0$, then $y_i \geq 0$. In particular, if
      $x,y\neq 0$ then $x_jy_j >0$ for some index $j$.
  \end{lemma}
  \begin{proof}
      By \eqref{eq:dualpair} and the assumption we have
      \begin{equation*}
        \|x\| \, \|y\|^* = \langle x,y \rangle = \sum_{x_i>0} x_i y_i \leq
        \sum_{x_i>0} x_i |y_i| =\langle x,|y| \rangle \leq \|x\| \, \|\,|y|\,\|^*  \,.
      \end{equation*}
      As $\|\cdot\|^*$ is absolute, we have $\|y\|^*=\|\,|y|\,\|^*$ and so
      equality throughout. This implies that all summands in $\sum_{x_i>0}
      x_i y_i$ are nonnegative, which is the first assertion. The second
      claim follows as by assumption $\|x\|\|y\|^* > 0$ and so some of the
      terms in the sum $\sum_{x_i>0} x_i y_i$ have to be positive.
\hfill~\qed
  \end{proof}

  It follows in particular, that if $\| \cdot\|$ is an absolute norm, and
  $x\gg 0$, then any dual vector $y$ to $x$ satisfies $y\geq 0$.  A
  particular absolute norm is the usual infinity norm denoted by
  $\|\cdot\|_{\infty}$.

{\em Monotone Systems}

Throughout the paper, $\mathcal{W}$ is a neighbourhood of
$\mathbb{R}^n_+$.  Let $f: \mathcal{W} \rightarrow \RR^n$ be a $C^1$
vector field.  The \emph{forward solution} of the nonlinear system
\begin{equation}\label{eq:sys1}
\dot{x}(t) = f(x(t)), \quad\; x(0) = x^0.
\end{equation}
with initial condition $x^0\in \WW$ at $t=0$ is denoted by $x(t,x^0)$ and
is defined on the maximal forward interval of existence
$\mathcal{I}_{x^0}:=[0, T_{max}(x^0))$.  All autonomous nonlinear systems
considered here have unique solutions defined on $[0, \infty)$.

The systems considered here are \emph{positive systems}. A system is
called positive if $x^0 \geq 0$ implies $x(t, x^0) \geq 0$ for all $t \geq
0$.  It is well known \cite{Del00} that the following property is necessary
and sufficient for (\ref{eq:sys1}) to be positive
\begin{equation}\label{eq:positivity}
\forall x \in \bd (\RR^n_+):x_i=0 \Rightarrow f_i(x)\geq 0.
\end{equation}
In particular, a linear system $\dot{x}=Ax$ is positive if and only if $A$
has nonnegative off-diagonal entries, i.e.\ $a_{ij}\geq0$, $\forall i\neq
j$.  Matrices with this property are called {\em Metzler}, or sometimes negative
$Z$-matrices or essentially nonnegative matrices. In this paper we stick to the former terminology.

As is standard, we say that the $C^1$ vector field $f:\mathcal{W}
\rightarrow \RR^n$ is \textit{cooperative} on $\UU \subseteq \WW$ if the
Jacobian matrix $\frac{\partial f}{\partial x}(a)$ is Metzler for all
$a\in \UU$.  When we say that $f$ is cooperative without specifying the
set $\UU$, we understand that it is cooperative on $\RR^n_+$.  It is well
known that cooperative systems are monotone \cite{Smi95}.  I.e., if $f:\WW \rightarrow \RR^n$ is cooperative on $\RR^n_+$ then
$x^0 \leq y^0$, $x^0, y^0 \in \RR^n_+$ implies $x(t, x^0) \leq x(t, y^0)$
for all $t \geq 0$.

We will use the following lemma, which is an immediate consequence of
\cite[Prop.~3.2.1]{Smi95}.
\begin{lemma}\label{lem:nonincreas}
    Let $f:\WW \rightarrow \RR^n$ be cooperative and $w\in \WW $ be such
    that $f(w)\ll 0$ ($f(w)\gg 0$). Then the trajectory $x(\cdot,w)$ of system
    (\ref{eq:sys1}) is decreasing (increasing) in $t$ for $t\geq 0$ with
    respect to the order on $\RR^n_+$, i.e. if $f(w)\ll 0$, then
    \begin{equation*}
        x(t,w) \ll x(s,w) \,,\quad \text{for all}\quad 0\leq s < t < T_{{\rm max}}(w)\,.
    \end{equation*}
    If $f(w)\leq 0$ ($f(w)\geq 0$), the trajectory will be
    non-increasing (non-decreasing).
\end{lemma}

\textit{Switched Systems}

The main contributions of this manuscript extend results for
time-invariant compartmental SIS models to switched SIS models.  We now
briefly recall some fundamental concepts related to switched nonlinear
systems, \cite{Lib03, Sho07}. Let $\WW$ be an open neighbourhood of
$\RR^n_+$. Consider a family $\{f_1, \ldots , f_m\}, m\in\N$, where
$f_i:\mathcal{W} \rightarrow \mathbb{R}^n$ is $C^1$, $1 \leq i\leq m$.  We
assume that the associated autonomous systems $\dot{x} = f_i(x)$, are
forward complete on $\RR^n_+$, that is, the unique solution $x(\cdot, x^0)$ is
defined on $[0,\infty)$ for all $x^0 \in \mathbb{R}^n_+$, $1 \leq i
\leq m$.

Associated to the family $\{f_1, \ldots , f_m\}$ and a set of switching
signals $\sigma:\RR_+ \to \{1,\ldots,m\}$ we consider the switched system
\begin{equation}\label{eq:switched}
\dot{x}(t)=f_{\sigma(t)}x(t) \quad \text{a.e.}\quad \; x(0) = x^0\,.
\end{equation}

By $\overline{\mathcal{S}}$ we denote the set of measurable functions
$\sigma(\cdot): \RR_+ \mapsto \{1, \cdots, m\}$. We note that for each $\sigma
\in \overline{\mathcal{S}}$ the Carath{\'e}odory conditions are satisfied,
so that existence and uniqueness of solutions is guaranteed,
\cite{CoddLevi55}. 
We refer to $\dot{x} = f_i(x)$ as the $i$th constituent
system for $1 \leq i \leq m$.

The subset $\mathcal{S}\subset \overline{\mathcal{S}}$ denotes the set of
all piecewise constant mappings $\sigma(\cdot): \RR_+ \mapsto \{1, \ldots,
m\}$ that are continuous from the right and for which there exists some
$\tau > 0$ such that $t - s \geq \tau$ for any two points of discontinuity
$t, s$ of $\sigma$.

Note that the set $\mathcal{S}\subset\overline{\mathcal{S}}$ is dense in
$\overline{\mathcal{S}}$, considered as subsets of $L^\infty(\RR_+,\RR)$
endowed with the weak$^*$ topology. As a consequence solutions of
\eqref{eq:switched} for $\sigma\in \overline{\mathcal{S}}$ may be
approximated, on any compact interval, arbitrarily well by solutions
corresponding to $\sigma\in {\cal S}$. Throughout the paper, we use the
notation $x(\cdot, x^0, \sigma)$ to denote the solution corresponding to
$\sigma \in \overline{\mathcal{S}}$ and the initial condition $x^0 \in
\mathbb{R}^n_+$.

The points of discontinuity $0 = t_1, t_2, \ldots $ of $\sigma\in{\cal S}$
are the \textit{switching instants}.

By assumption for every $\sigma\in{\cal S}$, there is some positive
constant $\tau$ such that $t_{i+1} - t_i \geq \tau$ for all $i$; for
switched SIS models, where switching corresponds to some seasonal or
diurnal change, this is entirely reasonable.

We next recall various fundamental stability concepts.  As we are
dealing with positive systems, all definitions
are with respect to the state space $X = \RR^n_+$.
\begin{definition}\label{def:stability}
    Let $\sigma \in \overline{\mathcal{S}}$ be given and let $\bar{x}$ be
    an equilibrium of the system (\ref{eq:switched}). Then we say that the
    equilibrium point $\bar{x}$ is
\begin{enumerate}[(i)]
  \item \emph{stable}, if for every $\varepsilon>0$, there exists a
    $\delta=\delta(\varepsilon)>0$ such that
\begin{equation}\nonumber
    \|x^0-\bar{x}\|<\delta \Rightarrow \|x(t,x^0, \sigma)-\bar{x}\|<\varepsilon,
    \quad \forall t>0.
\end{equation}
\item \emph{unstable}, if it is not stable;
\item \emph{asymptotically stable} if it is stable and there exists a
  neighbourhood $N$, relatively open in $\RR^n_+$, of $\bar{x}$ such that
\begin{equation}\nonumber
 x^0 \in N \Rightarrow \lim_{t\rightarrow \infty}x(t, x^0, \sigma)=\bar{x}.
\end{equation}
\end{enumerate}
The \textit{domain of attraction} of $\bar{x}$ is given by
\[A(\bar{x}):= \{x^0 \in \RR^n_+: x(t, x^0, \sigma)
\rightarrow \bar{x}, \mbox{ as } t \rightarrow \infty \}\,.\]
If $A(\bar{x}) =
\RR^n_+$, then we say that $\bar{x}$ is globally asymptotically
stable.
\end{definition}
When discussing uniform stability with respect to a set of switching signals, it is convenient to make use of  $\mathcal{K}$ and $\mathcal{KL}$ functions (the definitions given above for a fixed $\sigma$ can also be formulated equivalently in these terms).  A function $\alpha : \RR_+ \rightarrow \RR_+$ is of class $\mathcal{K}$ if $\alpha(0) = 0$ and $\alpha$ is strictly increasing.  A function $\beta: \RR_+ \times \RR_+ \rightarrow \RR_+$ is of class $\mathcal{KL}$ if $\beta(\cdot, t)$ is of class $\mathcal{K}$ for every fixed $t \geq 0$ and $\beta(r, t) \rightarrow 0$ as $t \rightarrow \infty$ for each fixed $r \geq 0$.

We say that $\bar{x}$ is a uniformly stable equilibrium of (\ref{eq:switched}) with respect to $\mathcal{S}$ ($\overline{\mathcal{S}}$) if there exists a class $\mathcal{K}$ function $\alpha$ such that
\begin{equation}
\|x(t, x^0, \sigma) - \bar{x} \| \leq \alpha(\|x^0 - \bar{x}\|) \quad \forall t \geq 0
\end{equation}
for all $x^0 \in \RR^n_+$ and all $\sigma \in \mathcal{S}$ ($\sigma \in \overline{\mathcal{S}}$).

In addition, $\bar{x}$ is a uniformly globally asymptotically stable equilibrium of (\ref{eq:switched}) with respect to $\mathcal{S}$ ($\overline{\mathcal{S}}$) if there exists a class $\mathcal{KL}$ function $\beta$ such that
\begin{equation}
\|x(t, x^0, \sigma) - \bar{x} \| \leq \beta(\|x^0 - \bar{x}\|, t) \quad \forall t \geq 0
\end{equation}
for all $x^0 \in \RR^n_+$ and all $\sigma \in \mathcal{S}$ ($\sigma \in \overline{\mathcal{S}}$).  They key concept here is that the rate of convergence as $t \rightarrow \infty$ is uniform across all switching signals in $\mathcal{S}$ ($\overline{\mathcal{S}}$), where this uniformity is guaranteed by the same $\mathcal{KL}$ function $\beta$ for all $\sigma$ in the set of signals.

Finally, we need some formulations of Lipschitz continuous Lyapunov
functions that will be used in later proofs. Let $D\subset \RR^n$ be open
and $h:D\to \RR$ be locally Lipschitz continuous. By Rademacher's theorem
this implies that $h$ is differentiable almost everywhere. The {\em Clarke
  generalized gradient} of $h$ at $x\in D$ may then be defined by, \cite{ClarLedy98}, \cite[Theorem~II.1.2]{DemyRubi95}, 
\begin{equation}
    \label{eq:Clarkedef}
    \partial_C h(x) = \conv \left\{ p  \in \RR^n \midset \exists x_k \to x, \triangledown h(x_k) \text{ exists }, \lim_{k\to\infty} \triangledown h(x_k) = p \right\}\,.
\end{equation}
In particular, it is an exercise to see that the Clarke gradient of a norm
$v$ at a point $x\neq 0$ is given by normed dual vectors, i.e.
\begin{equation}
    \label{eq:dualClarke}
    \partial_C v(x) = \{ y \in \RR^n  \midset v^*(y) = 1, y \text{ is dual to } x \}\,.
\end{equation}
A proper, positive definite, Lipschitz continuous function $V$ is a strict
Lyapunov function for the switched system \eqref{eq:switched}, if
\begin{equation}
    \label{eq:Lyapcond}
    \langle p , f_j(x) \rangle \leq -\alpha(\|x\|)
\end{equation}
for all $p\in \partial_C V(x)$ and $j=1,\ldots,m$ and some positive
definite function $\alpha:\RR_+\to \RR_+$. If a strict Lyapunov function
exists, this implies uniform asymptotic stability of the origin
with respect to $\sigma\in\overline{S}$. If in \eqref{eq:Lyapcond}
$-\alpha(\|x\|)$ has to be replaced by $0$, then we speak of a nonstrict
Lyapunov function, the existence of which implies uniform stability,
\cite[Theorem~4.5.5]{ClarLedy98}.

\section{Problem Description}\label{sec:problem}
{ We now recall some results concerning the basic autonomous SIS
model considered in \cite{FIST07}.

The population of interest is assumed to be divided into $n$ groups; each
group is then further divided into two classes: infectives and
susceptibles.  Let $I_i(t)$ and $S_i(t)$ be the number of infectives
and susceptibles at time $t$ in group $i$ for $i=1, \ldots, n$,
respectively. Also, let $N_i(t)=S_i(t)+I_i(t)$ be the total
population of group $i$.  The total population of each group is
assumed constant: formally, $N_i(t)=N_i,  \forall \ t \geq 0$.

The constants $\beta_{ij}$ denote the rate at which susceptibles in group $i$ are
infected by infectives in group
 $j$ for $i,j=1, \cdots, n$.  Further, $\gamma_i$ is the rate at which an infective individual in group
  $i$ is cured.  As the total population of each group is constant, within each group the birth
  and death rates are equal and denoted by $\mu_i$.  Using the mass-action law, the basic
   SIS model is described as follows \cite{FIST07}:
\begin{equation}\nonumber 
\left\{ \begin{array}{l}
\dot{S}_i(t) = \mu_i N_i - \mu_i S_i(t)
-\displaystyle \sum_{j=1}^n{\beta_{ij} \frac{S_i(t) I_j(t)}{N_i}} + \gamma_i I_i(t) \\
\dot{I}_i(t) = \displaystyle \sum_{j=1}^n{\beta_{ij}
\frac{S_i(t)I_j(t)}{N_i}} - (\gamma_i+\mu_i) I_i(t)
\end{array} \right.
\end{equation}
Since the population of each group is constant, it is sufficient to
know $I_i(t)$. If we set $x_i(t)=I_i(t)/N_i$ and
$\tilde{\beta}_{ij}=\beta_{ij} N_j/N_i$ and $\alpha_i=\gamma_i+\mu_i$,
we obtain
\begin{equation}\nonumber 
\dot{x}_i(t)=(1-x_i(t))\displaystyle \sum_{j=1}^n{\tilde \beta_{ij}
x_j(t)} -\alpha_i x_i(t),
\end{equation}
which can be written in the compact form:
\begin{equation}\label{eq:Fall}
\dot{x}=[-D+B-\diag(x)B]x,
\end{equation}
where $D=\diag(\alpha_i)$ and $B=(\tilde{\beta}_{ij})>0$. This is
the model considered in \cite{LY76, FIST07}.  In \cite{FIST07} and throughout this paper, we assume that $\alpha_i > 0$ for $1 \leq i \leq n$.  This is biologically reasonable, as otherwise there would exist a compartment for which both the birth-rate and the rate at which infectives are cured were zero.  This assumption has implications for the location of so-called endemic equilibria on which we shall elaborate below.

The following properties of (\ref{eq:Fall}) are easy to check.
\begin{itemize}
  \item[(i)] The compact set $\Sigma_n:=\{x \in \mathbb{R}^n_+: x_i \leq 1,
    i=1,\ldots,n\}$ is forward invariant under (\ref{eq:Fall}).
\item[(ii)] The origin $0$ is an equilibrium point of (\ref{eq:Fall}).
This is referred to as the \emph{Disease Free Equilibrium} (DFE) of (\ref{eq:Fall}).
\item[(iii)] For every $x^0 \in \Sigma_n$, there exists a unique solution $x(t, x^0)$
 of (\ref{eq:Fall}) defined for all $t \geq 0$.
\end{itemize}
If there exists an equilibrium point $\bar{x}$ in
$\intt(\mathbb{R}^n_+)$, we refer to $\bar{x}$ as an \emph{endemic
equilibrium}.  Note that as $\alpha_i > 0$ for $1 \leq i \leq n$, it follows readily that any endemic equilibrium must lie in $\intt(\mathbb{R}^n_+) \cap \intt(\Sigma_n)$.

The basic reproduction number $R_0$, defined as the average number
of secondary infections that occur when one infective is introduced
into a completely susceptible host population, is fundamental in
mathematical epidemiology.  In line with \cite{DW02}, in
\cite{FIST07} this is defined as $R_0=\rho(D^{-1}B)$.  From basic
properties of Metzler matrices it is easy to see that $R_0 < 1$ if
and only if $\mu(D + B) < 0$ \cite{HornJohn}.

The parameter $R_0$ is used to characterize the existence and
stability of the equilibria of (\ref{eq:Fall}).  The following
result is Theorem 2.3 in \cite{FIST07}.
\begin{theorem}\label{thm:Fall1}
Consider the system (\ref{eq:Fall}).  Assume that the matrix $B$ is
irreducible.  The DFE at the origin is globally asymptotically
stable if and only if $R_0\leq 1$.
\end{theorem}

The next result considers the existence and stability of endemic
equilibria and is a restatement of Theorem 2.4 of \cite{FIST07}.
\begin{theorem}\label{thm:Fall2}
Consider the system (\ref{eq:Fall}) and assume that $B$ is
irreducible.  There exists a unique endemic equilibrium $\bar{x}$ in
$\intt(\RR^n_+)$ if and only if $R_0 > 1$.  Moreover, in this case,
$\bar{x}$ is asymptotically stable with region of attraction
$\RR^n_+\setminus \{0\}$.
\end{theorem}

} 
Note that as $\Sigma_n$ is forward invariant under (\ref{eq:Fall}), it follows readily that
the unique endemic equilibrium discussed in Theorem \ref{thm:Fall2} must be in $\Sigma_n$.  Furthermore,
as $\alpha_i > 0$ for $1 \leq i \leq n$, it follows readily that it must
lie in $\intt(\mathbb{R}^n_+) \cap  \Sigma_n$.

\section{Positive linear switched systems}
\label{sec:PosLinSys}

To facilitate our later analysis of the linearisation of a switched
version of the SIS model \eqref{eq:Fall} we need some results on positive
switched linear systems, which are not available in the literature. This
is the topic of the present section. The material presented here is a
consequence of results presented in \cite{ABMW12}. As it is essential for
the following arguments we include it for the benefit of the reader.

We make use of a formulation of the joint
spectral radius for continuous time systems as described in \cite{Wirt02}.
Let $\MM \subset \RR^{n \times n}$ be a compact set of matrices. In
particular, a finite set ${\cal M}=\{A_1,\cdots, A_m\}$ may be considered.
The set ${\cal M}$ gives rise to the switched linear system
\begin{equation}
    \label{eq:swlinsys}
    \dot x (t) = A_{\sigma(t)} x(t) \,,\quad \sigma \in \overline{{\cal S}}\,.
\end{equation}

If we fix $\sigma:\RR_+\to {\cal M}$, then the evolution operator
$\Phi_\sigma(\cdot)$ (with initial time $t_0=0$) corresponding to
\eqref{eq:swlinsys} is given as the solution of the matrix differential
equation
\begin{equation*}
    \dot \Phi_\sigma(t) = A_\sigma(t) \Phi_\sigma(t)
     \,,\quad \Phi_\sigma(0) = I\,.
\end{equation*}
Considering the set of all evolution operators we define a matrix
semigroup $\mathcal{H}$ as follows. The set of time $t$ evolution
operators is given by
\begin{equation*}
    {\cal H}_t := \{ \Phi_\sigma(t) \midset \sigma: [0,t] \to {\cal M} \text{ measurable} \}
\end{equation*}
and ${\cal H}:= \bigcup_{t\in\RR_+} {\cal H}_t$, where we set ${\cal H}_0
:=\{I\}$. Define the growth at time $t$ by
\begin{eqnarray*}\nonumber
    \rho_t(\MM):=\sup_{\sigma\in \overline{{\cal S}}}\ \frac{1}{t} \log
\|\Phi_\sigma(t)\| \,.
\end{eqnarray*}
Then the {\em joint Lyapunov exponent} of \eqref{eq:swlinsys} is
given by
\begin{equation}\label{eq:JSR}
\rho(\MM):=\displaystyle \lim_{t \rightarrow
\infty} \ \rho_t(\MM)\,.
\end{equation}
We note that the definition of $\rho({\cal M})$ does not depend on the
choice ${\cal S}$ versus $\overline{{\cal S}}$. Similarly to the {\em
  joint spectral radius} of a set of matrices, the joint Lyapunov exponent
can be used to characterise uniform asymptotic stability of linear
switched systems; see \cite{Wirt02} and references therein.  For our
purposes the following fact is sufficient.

\begin{lemma}\label{lem:JSR}
The joint Lyapunov exponent $\rho({\cal M})<0$ if and only if
the origin is a uniformly asymptotically stable
equilibrium of \eqref{eq:swlinsys}.
\end{lemma}

Also while it is necessary for the uniform exponential stability of
\eqref{eq:swlinsys} that all the matrices in the convex hull $\conv({\cal
  M})$ are Hurwitz, this is not a sufficient condition, even if ${\cal M}$
consists of Metzler matrices \cite{FainMarChig,GurShoMas}.

In the analysis of linear inclusions extremal norms play an interesting
role.

\begin{definition}{}
    \label{d:exnorm}
    Let ${\cal M}$ be a compact set of Metzler matrices.  A
    norm $v$ on $\RR^n$ is called extremal for the associated semigroup
    ${\cal H}$, if for all $x\in \RR^n$ and all $t\geq0$ we have
    \begin{equation}
        \label{eq:exnormineq}
        v(Sx) \leq \exp(\rho({\cal M})t)\ v(x)
      \,,\quad \forall\ S\in {\cal H}_t\,.
    \end{equation}
\end{definition}

For the analysis of the switched SIS model \eqref{eq:Fall} it will be
instrumental to analyse the existence of extremal norms for its
linearisation. We will show the more general statement that for positive
switched linear systems an absolute extremal norm exists, if an
irreducibility property is satisfied.

Associated to a set of Metzler matrices ${\cal M}$ we
consider the directed graph ${\cal G}({\cal M})=(V,E)$ with vertex set
$V=\{1,\ldots,n\}$ and edges defined for $i\neq j, 1\leq i,j\leq n$ by
\begin{equation}
    \label{eq:nodedef}
    (i,j) \in E \quad :\Leftrightarrow \quad \exists A\in{\cal M} \text{ with }
    a_{ij} > 0\,.
\end{equation}
Note that we explicitly do not define edges from a vertex $i$ to
itself. The graph ${\cal G}({\cal M})$ is called strongly connected if for
all $i,j$ there is a path form $i$ to $j$ using edges in $E$.

\begin{lemma}
    \label{lem:connectedchar}
    Let ${\cal M}$ be a set of Metzler matrices. Then ${\cal G}({\cal M})$
    is strongly connected if and only if for all $i,j\in V$ there exist
    $k\in \N$, $(A_1,\ldots,A_k) \in {\cal M}^k$ and
    $t_1,\ldots,t_k> 0$ such that
        \begin{equation*}
           \left(e^{A_1t_1}\dots e^{A_kt_k} \right)_{ij} >0\,.
        \end{equation*}
\end{lemma}

\begin{proof}
    Follows by direct calculation.\hfill~\qed
\end{proof}

We also find it useful to point out the following fact.

\begin{lemma}
    Let ${\cal M}$ be a set of Metzler matrices. Then ${\cal
      G}({\cal M})$ is strongly connected if and only if the convex hull
    $\conv {\cal M}$ contains an irreducible matrix.
\end{lemma}

We are now ready to formulate our main result for positive switched linear
systems.

\begin{proposition}
    \label{p:metextr}
    Let ${\cal M}$ be a compact set of Metzler matrices. If $\conv {\cal
      M}$ contains an irreducible element $\bar{M}$ then there exists a
    absolute extremal norm $v$ for \eqref{eq:swlinsys}.
\end{proposition}

\begin{proof}
    By convexity of norms, a norm $v$ is extremal for ${\cal M}$ if and
    only if it is extremal for $\conv {\cal M}$, so that we may assume
    that ${\cal M}$ is convex. By considering ${\cal M} - \rho({\cal M})
    I$ we may assume that $\rho({\cal M})=0$.

    We first show that under the assumptions ${\cal H}$ is bounded if
    $\rho({\cal M})=0$. Assume this is not the case. As ${\cal M}$ is
    irreducible and convex, there exists an irreducible matrix
    $\bar{M}\in{\cal M}$. Then $\exp(\bar{M})\in {\cal H}_1$ is positive,
    \cite{BP87}, and there exists a constant $c>0$ such that for all
    $i,j=1,\ldots,n$
    \begin{equation}
        \label{eq:ijestimate}
        \exp(\bar{M})_{ij} > c \,.
    \end{equation}

    As ${\cal H}$ is assumed to be unbounded, we may choose
    $S=\Phi_\sigma(t,0)\in{\cal H}_t$ such that $S_{\nu\mu} > 1/c$ for
    some indices $\nu,\mu$.  It follows by direct calculation that
    $\bigl(\exp(\bar{M})S\bigr)_{\mu\mu} =: a > 1$ and as
    $\bigl(\exp(\bar{M})S\bigr)^k \in {\cal H}_{k(t+1)}$ we obtain
    $\rho({\cal M}) \geq \log (a)/(t+1) >0$. This contradiction proves of boundedness of ${\cal H}$. 

    As the generated semigroup is bounded, an extremal norm may be
    defined in the following way, see also \cite{Kozy90}.  Let $\|\cdot\|$
    be absolute, then define for $x\geq 0$
    \begin{equation}
        \label{eq:defextnorm}
        v(x) := \sup \{ \| S x\| \midset S\in {\cal H} \} \,.
    \end{equation}
    and extend this definition to $\RR^n$ by setting $v(x):= v(|x|), x\in
    \RR^n$. It is clear that $v$ is positively homogeneous and positive
    definite (as $I\in{\cal H}$). As the matrices $S\in{\cal H}$ are all
    nonnegative it follows by absoluteness of $\|\cdot\|$ that for $0\leq
    x<y$ we have $v(x) < v(y)$. As a consequence $v$ is absolute because
    if $|x| < |y|$ then $v(x) = v(|x|) < v(|y|) = v(y)$.

The triangle inequality for $v$ then follows from
    \begin{align*}
        v(x+y) = v(|x+y|) \leq v(|x|+|y|) = \sup \{ \| S (|x| + |y|)\| \midset S\in {\cal H} \} \\ \leq \sup \{ \| S |x|\| +\|S |y|)\| \midset S\in {\cal H} \} \leq v(|x|) + v(|y|) = v(x) + v(y)\,.
    \end{align*}
    Thus $v$ is a norm.  
    Further, from $0\leq|Sx| \leq S|x|$, using the monotonicity of $v$ (and the
    assumption $\rho({\cal M})=0$) we have for all $x\in \RR^n, S\in{\cal
      H}$ that
    \begin{equation}
        \label{eq:df2}
        v(Sx) = v(|Sx|) \leq v(S|x|) = \sup \{ \| TS |x| \| \midset T\in {\cal H} \} \leq v(|x|) = v(x)\,,
    \end{equation}
 where we have used the definition of $v$ in the final
   inequality. Now extremality of $v$ follows.\hfill~\qed
\end{proof}

We note the following consequence for stability theory of switched
  positive systems.

  \begin{corollary}
      \label{c:nonstrictLyap}
      Let ${\cal M}$ be a compact set of Metzler matrices and consider the
      positive switched linear system \eqref{eq:swlinsys}.  The following
      two conditions are each sufficient for the existence of an absolute
      norm $v$, which is a nonstrict Lyapunov function for
      \eqref{eq:swlinsys}.
      \begin{enumerate}[(i)]
        \item $\rho({\cal M}) = 0$ and $\conv {\cal M}$ contains an
          irreducible element,
        \item $\rho({\cal M}) < 0$.
      \end{enumerate}
  \end{corollary}
  \begin{proof}
      We note that $v(\Phi_\sigma(t)x) \leq v(x)$ for all $x\in \RR^n,
      t\geq 0 $ and all $\sigma \in {\cal S}$ is equivalent to the
      statement that $\langle y, A x \rangle \leq 0$ for all dual pairs
      $(y,x)$ and all $A \in {\cal M}$, see the comments after
      \eqref{eq:Lyapcond}.  The case (i) is then immediate from the
      previous Proposition~\ref{p:metextr}. In case (ii) holds, we have
      that the switched linear system has a uniformly asymptotically
      stable equilibrium at $x^\ast =0$ by Lemma~\ref{lem:JSR}. In
      particular, the semigroup ${\cal H}$ is bounded and we can without
      further ado use \eqref{eq:defextnorm} to define an absolute norm. It
      follows as in \eqref{eq:df2} that the norm is a nonstrict Lyapunov
      function.~\hfill\qed
  \end{proof}

  We note that in the case (ii) of the previous corollary it may or may
  not be possible to construct an extremal norm for ${\cal M}$.

\section{Global Asymptotic Stability of the Disease Free Equilibrium}\label{sec:DFE}
In this section, we present an extension of Theorem~\ref{thm:Fall1}
to a switched SIS model.  In place of $R_0$, we use the joint
Lyapunov exponent to characterize the stability of the Disease
Free Equilibrium (DFE).  In \cite{BMW10}, conditions for local asymptotic
stability of the DFE based on the joint Lyapunov exponent (joint
  spectral radius) were presented.  Our result here shows that the same
conditions imply global asymptotic stability for all switching signals.

For the remainder of the paper, $D_1, \ldots, D_m$ are diagonal
matrices in $\mathbb{R}^{n \times n}$ with positive entries along
their main diagonals; $B_1, \ldots , B_m$ are nonnegative matrices
in $\mathbb{R}^{n \times n}$ and $\mathcal{S}$ denotes the set of
admissible switching signals introduced in Section
\ref{sec:bkground}.

We now consider the switched SIS system described by
\begin{equation}\label{eq:Fallsw}
\dot{x}=(-D_{\sigma(t)}+B_{\sigma(t)}-\diag(x)B_{\sigma(t)})x =: f_{\sigma(t)}(x(t)).
\end{equation}
It follows easily from the properties of the constituent systems,
outlined in Section~\ref{sec:problem}, that for each $x^0 \in \Sigma_n$,
$\sigma \in \overline{\mathcal{S}}$, there is a unique solution $x(t, x^0,
\sigma)$ of (\ref{eq:Fallsw}) defined for $t \geq 0$, satisfying
$x(0, x^0, \sigma) = x^0$. Further, it is clear that the set $\Sigma_n$
is forward invariant under \eqref{eq:Fallsw} for all $\sigma \in
\overline{\mathcal{S}}$.

Linearising (\ref{eq:Fallsw}) about the origin, we obtain the linear
switched system
\begin{equation}\label{eq:Falllinsw}
\dot{x}=(-D_{\sigma(t)}+B_{\sigma(t)})x .
\end{equation}
Let $\mathcal{M} = \{-D_1 + B_1, \ldots , -D_m + B_m\}$ and recall
that $\rho(\mathcal{M})$ denotes the joint Lyapunov exponent of
$\mathcal{M}$.  An application of Lemma
\ref{lem:JSR} to \eqref{eq:Falllinsw} yields
\begin{lemma}\label{lem:genspeclin}
    The joint Lyapunov exponent $\rho(\MM)<0$ if and only if the origin is
    a uniformly globally asymptotically stable equilibrium of
    (\ref{eq:Falllinsw}) w.r.t. $\sigma \in \mathcal{S}$ (or
    $\sigma\in\overline{{\cal S}}$).
\end{lemma}

We note the following comparison properties of system \eqref{eq:Fallsw}
and the linearisation of the system in $x=0$ given by \eqref{eq:Falllinsw}.

\begin{lemma}
\label{lem:comparison}
Consider \eqref{eq:Fallsw} and its linearisation
\eqref{eq:Falllinsw}. Let $0\leq x^0 \leq y^0$ and let
$\varphi(t):=\varphi(t;x^0)$, resp, $\Phi_\sigma(t,0)y^0$ be the solutions
of \eqref{eq:Fallsw} and \eqref{eq:Falllinsw}. Then
    \begin{equation}
        \label{eq:majorization}
        \varphi(t;x^0) \leq  \Phi_\sigma(t,0)y^0\,,\quad \forall\ t\geq 0\,.
    \end{equation}
\end{lemma}
\begin{proof}
    The claim is immediate using variation of constants and the
    nonnegativity of $\Phi_\sigma$ and $B_\sigma$ which yields
    \begin{align*}
        \varphi(t;x^0) &= \Phi_\sigma(t,0)x^0 - \int_0^t \Phi_\sigma(t,s) \diag(\varphi(s))B_{\sigma(s)}\varphi(s) ds\\
  &\leq \Phi_\sigma(t,0)x^0 \leq
\Phi_\sigma(t,0)y^0 \,. 
    \end{align*}~\hfill~\qed
\end{proof}

The following theorem generalizes Theorem~\ref{thm:Fall1} to the case of
switched systems and extends Theorem~\ref{thm:Fall1} even for the case of
an autonomous system, i.e. for the situation considered in
Theorem~\ref{thm:Fall1}.

\begin{theorem}
    \label{t:stabimpliesasstab}
    Assume that the linear switched system \eqref{eq:Falllinsw} is
    uniformly stable in $x^*=0$ and admits an absolute norm as a
    nonstrict Lyapunov function. Then
    the disease free equilibrium $x^*=0$ of \eqref{eq:Fallsw} is uniformly
    globally asymptotically stable for $\sigma\in\overline{{\cal S}}$.
\end{theorem}

\begin{remark}
    (i)~In the case $\rho({\cal M})=0$, by
      Corollary~\ref{c:nonstrictLyap}, the linear switched system
    \eqref{eq:Falllinsw} admits an absolute norm as a nonstrict
      Lyapunov function, provided an appropriate irreducibility condition
    holds. So just as in the case of Theorem~\ref{thm:Fall1}
    irreducibility and stability of the linearisation imply global
    asymptotic stability of the nonlinear switched system. This is in
    contrast to the general linearisation theory in which exponential
    stability of the linearisation, equivalently $\rho({\cal M})<0$, is
    required in order to infer
    stability of the nonlinear system.\\
    (ii)~Observe that most of the work in the proof of
      Theorem~\ref{t:stabimpliesasstab} is devoted to the case that the
      linearization is not uniformly asymptotically stable. The
      asymptotically stable case can be proved in a much simpler fashion
      by applying Lemma~\ref{lem:comparison}.\\
    (iii)~Theorem~\ref{t:stabimpliesasstab} also applies to autonomous
    systems: the case $m=1$. Even in this case it generalizes
    Theorem~\ref{thm:Fall1}: If the basic reproduction number $R_0=1$,
    then irreducibility of $B$ implies the existence of an extremal norm,
    but the converse is false. Still under the assumption that
      $\mu(A)=0$, it is not necessary that a Metzler matrix $A$
    is irreducible for an absolute extremal norm to exist. The necessary
    and sufficient condition for this is that the matrix is stable,
    i.e. that for eigenvalues $\lambda\in \sigma(A), \mathrm{Re}(\lambda)
    =0$, algebraic and geometric multiplicity coincide. While this
      is well known in the context of quadratic Lyapunov functions, we
      point out that the construction in \eqref{eq:defextnorm} yields an
      absolute extremal norm also in this case. The following corollary
    is thus
    immediate.
\end{remark}

\begin{corollary}
    \label{c:fall-complete}
    Consider the time-invariant system \eqref{eq:Fall}. The
    disease free equilibrium $x^*=0$ is globally asymptotically stable, if
    the matrix $-D+B$ is stable.
\end{corollary}

\begin{proof} (of Theorem~\ref{t:stabimpliesasstab}) If
    \eqref{eq:Falllinsw} is stable, then a nonstrict Lyapunov function
    $v(\cdot)$ on $\RR^n$ has the property that for all switching signals
    $\sigma \in \overline{{\cal S}}$ and all initial conditions we have
    \begin{equation}
        \label{eq:linlyap}
        v(\Phi_\sigma(t,0)x^0) \leq v(x^0) \,,\quad \forall \ t\geq0\,.
    \end{equation}
    {\bf Step 1:} We first show that $v$ is a nonstrict Lyapunov
    function for \eqref{eq:Fallsw} on $\RR^n_+$.

    Consider a vector $x\in \RR^n_+, x\neq 0$ and a dual vector $y$ with
    $v^*(y)=1$ for $x$ and fix a constituent system given by $(B_j,D_j),
    j=1,\ldots,m$. As $v$ is an absolute norm we have using
    \eqref{eq:linlyap} and \eqref{eq:dualClarke} that
     \begin{equation}
         \label{eq:extnormFall}
         \langle y , (-D_j + B_j)x\rangle \leq 0\,.
     \end{equation}
     Further, as $x_i> 0$ implies $y_i\geq0$ by Lemma~\ref{lem:dualabs} it
     follows that
     \begin{equation}
         \label{eq:dual0}
       \langle y, - \diag(x) B_j x\rangle \leq 0,
      \end{equation}
      so that we always have
     \begin{equation}
         \label{eq:dual1}
         \langle y, f_j(x) \rangle = \langle y , (-D_j + B_j)x - \diag(x) B_j x \rangle \leq 0\,.
     \end{equation}
     In particular, this shows that \eqref{eq:Fallsw} is uniformly stable
     with respect to $\sigma \in \overline{{\cal S}}$. It remains to show
     attractivity.

     {\bf Step 2:} We now investigate under which conditions
     \eqref{eq:dual1} may fail to be strict. By \eqref{eq:extnormFall} and
     \eqref{eq:dual0} this can only happen, if in both these equations
     equality holds. Assuming this we obtain from equality in
     \eqref{eq:dual0}
\begin{equation*}
  0= - \langle y, \diag(x) B_j x \rangle
         = - \sum_{x_iy_i\neq 0} y_i x_i (B_jx)_i  \,.
\end{equation*}
As $y_i\geq 0$, if $x_i>0$, this implies that from $x_i>0$ we can conclude
$y_i(B_jx)_i=0$. Plugging this into \eqref{eq:extnormFall} we obtain that
\begin{equation*}
    0= \langle y , (-D_j + B_j)x\rangle = \sum_{x_i>0} -x_iy_iD_{j,ii} + \sum_{x_i=0} y_i(B_jx)_i\,.
\end{equation*}
Again by Lemma~\ref{lem:dualabs} the first sum on the right hand side is
negative. We can conclude that the second sum has to be positive to
compensate this, so that there are $x_i=0$ with $(B_jx)_i>0$.

{\bf Step 3:} To exploit the property established in Step 2, that
solutions which are not infinitesimally decreasing have to move to the
interior of the positive orthant we add a further term to the Lyapunov
function. To this end let $\psi:\RR\to \RR$ be a $C^\infty$ function with
support contained in $(-\infty,1]$ and so that
\begin{equation*}
    \psi(0)= 1\,,\quad \psi(z) >0 \,,\ z\in [0,1)\,, \quad \psi'(z)<0\,,\ z \in
[0,1)\,,
\end{equation*}
and that in addition $\psi'$ is increasing on $[0,1]$.

For $\varepsilon\in(0,1)$ we denote $\psi_\varepsilon(z):=
\psi(z+(1-\varepsilon))$ and note that the support of $\psi_\varepsilon$
is contained in $(-\infty,\varepsilon]$ and that
$\eta(\varepsilon):=|\psi'_\varepsilon(0)|=\max_{z\in[0,\varepsilon]}|\psi_\varepsilon'(z)|
= \max_{z\in[1-\varepsilon,1]}|\psi'(z)|$
tends to $0$ as $\varepsilon\to 0$, because $\psi'(1)=0$.

Fix $0<\ell<L$. We aim to show that there exists
  $1>\varepsilon(\ell,L)>0$ so that for all $0< \varepsilon <
  \varepsilon(\ell,L)$ the function
\begin{equation}
    \label{eq:Vhatdet}
    V_\varepsilon(x) := v(x)\left(1 + \sum_{i=1}^n \psi_\varepsilon(x_i) \right) =
v(x)\left(1+ \sum_{x_i<\varepsilon} \psi_\varepsilon(x_i) \right)
\end{equation}
is a strict Lyapunov function for \eqref{eq:Fallsw} on the set $v^{-1}([\ell,L]):=\{
x\in \RR^n_+ \midset \ell\leq v(x) \leq L \}$. As $\ell>0$ may be chosen to be
arbitrarily small and $L>0$ arbitrary large, this shows that
the disease free equilibrium is globally asymptotically stable, as
  explained in Step 5 below.

{\bf Step 4:} We proceed to prove the claim formulated in the previous
step. Fix $x\in v^{-1}([\ell,L])$. We wish to show that 
  for $\varepsilon>0$ sufficiently small, we have for all elements of the
Clarke subgradient $p\in\partial_C V_\varepsilon(x)$ and $j=1,\ldots,m$
that
\begin{equation}
    \label{eq:CScond}
    \langle p , f_j(x) \rangle < -\theta < 0\,,
\end{equation}
where $\theta=\theta(x)>0$ is a suitable constant. As $\psi_\varepsilon$ is
smooth it is easy to see that the elements of $\partial_C V_\varepsilon(x)$ are of the form
\begin{equation}
    \label{eq:CPform}
    p=y \left(1+ \sum_{i=1}^n \psi_\varepsilon(x_i) \right) +
    v(x) \sum_{x_i<\varepsilon} \psi'_\varepsilon(x_i) e_i\,,
\end{equation}
where $e_i$ denotes the $i$th unit vector and $y\in \partial_C v(x)$,
i.e. $y$ is dual to $x$ and $v^*(y)=1$.

For the sake of estimation set $d_{\max}:= \max \{ D_{j,ii} \midset
j=1,\ldots,m, i=1,\ldots,n \}$. We now distinguish two cases. First, if
for fixed $j$ we have $\max \{ \langle y,f_j(x) \rangle \midset
y\in \partial_C v(x) \} = -c < 0$, then choose $1>\varepsilon>0$ so that
\begin{equation}
    \label{eq:comp1}
     L \, n\,  d_{\max}\, \eta(\varepsilon) \varepsilon < \frac{c}{2} \,.
\end{equation}
Then we obtain for $p\in\partial_C V_\varepsilon(x)$ given by \eqref{eq:CPform}
 that
\begin{align}
    \label{eq:CSconda}
    \langle p , f_j(x) \rangle &= \left(1+ \sum_{i=1}^n
      \psi_\varepsilon(x_i) \right) \langle y ,f_j(x) \rangle +
    v(x) \sum_{x_i<\varepsilon} \psi'_\varepsilon(x_i) \langle e_i , f_j(x)\rangle\\
    &\leq -c + v(x) \sum_{x_i<\varepsilon} \psi'_\varepsilon(x_i) \left(
      -D_{j,ii}x_i + (1-x_i)(B_jx)_i\right)\,. \nonumber \intertext{As
      $\psi'_\varepsilon(z)\leq0$ for $z\geq 0$ and the summands in the
      second term are only nonzero for $x_i<1$ we have by nonnegativity of
      $(B_jx)_i$} \nonumber
    & \leq -c + L \sum_{x_i<\varepsilon} \psi'_\varepsilon(x_i) \left( -D_{j,ii}x_i \right)\\
    & \leq -c + n\, L \, d_{\max}\, \eta(\varepsilon) \varepsilon < -
    \frac{c}{2} \,.
\end{align}

Secondly, consider the case $\max \{ \langle y,f_j(x) \rangle \midset
y\in \partial_C v(x) \} = 0$. Then we have seen in Step 2 that
for some index $k$ we have $x_{k}=0$ and $(B_jx)_k >0$. Choose
$1>\varepsilon>0$ small enough so that
\begin{equation}
    \label{eq:comp2}
     n  \, d_{\max}  \varepsilon  < (B_jx)_k \,.
\end{equation}
Then continuing from \eqref{eq:CSconda} we obtain
\begin{align}
    \label{eq:CSconda2}
    \langle p , f_j(x) \rangle & \leq v(x) \sum_{x_i<\varepsilon}
    \psi'_\varepsilon(x_i) \left( -D_{j,ii}x_i + (1-x_i)(B_jx)_i\right)\\
    \nonumber & \leq v(x) \left( \psi_\varepsilon'(0) (B_jx)_k +
      \sum_{x_i<\varepsilon, i\neq k} \psi'_\varepsilon(x_i) \left( -D_{j,ii}x_i +
        (1-x_i)(B_jx)_i\right)  \right)\\
    \intertext{and using that $-\psi'_\varepsilon(0)=\eta(\varepsilon)$} &
    \leq - \ell \,\eta(\varepsilon) \bigl( (B_jx)_k - n\, d_{\max}\,
    \varepsilon \bigr) < 0\,.
\end{align}

In both cases we have shown that we can satisfy the decrease condition
\eqref{eq:CScond} for all $j=1,\ldots,m$, all $x \in
  v^{-1}([\ell,L])$ and all $p\in \partial_C V_\varepsilon(x)$ by
choosing $1>\varepsilon>0$ small enough, but dependent on
$j,x$. In particular, note that the conditions \eqref{eq:comp1}
  and \eqref{eq:comp2} show that this conclusion holds for all
  $\varepsilon>0$ sufficiently small. Also, as the set-valued map $x
  \mapsto \partial_C V_\varepsilon(x)$ is upper semicontinuous with
  compact values, it follows that \eqref{eq:CScond} holds on an open
  neighbourhood of a point $x$. As there are finitely many constituent
systems and the set $v^{-1}([\ell,L])$ is compact, we can
use continuity and an open cover/compactness argument to conclude that
there is an $\varepsilon>0$ such that the decrease condition holds
uniformly for $V_\varepsilon$ for all $x  \in
  v^{-1}([\ell,L])$.

{\bf Step 5:} In oder to show uniform attractivity, note first
  that by construction $v(x) \leq V_\varepsilon(x) \leq (1+n) v(x)$.  Fix
$\delta>0$. Then for any $x^0\in \RR^n_+$, we may choose $L>0$ such that
$v(x) \leq L$, and $\varepsilon=\varepsilon(L,\delta)$ such that
$V_\varepsilon$ is a strict Lyapunov function on
$\{ x \midset \delta \leq v(x) \leq L
\}$. Let $c>0$ be the uniform constant of decay on this set, i.e. a
uniform bound $\langle p , f_j(x) \rangle \leq -c$ for all $x\in
  v^{-1}([\delta,L])$ and $p\in\partial_C V_\varepsilon(x)$. Then for any
switching signal $\sigma\in \overline{{\cal S}}$ we have
$V_\varepsilon(\varphi(t,x^0,\sigma)) \leq V_\varepsilon(x^0) -c
  t$ as long as $v(\varphi(t,x^0,\sigma))\geq \delta$. It follows that
uniformly $v(\varphi(t,x^0,\sigma))\leq (n+1)\delta$ for all $t\geq
(V_\varepsilon(x^0)-\delta)/c$ and all $\sigma$. As $\delta>0$ and $x^0$ are arbitrary
this shows uniform attractivity.  This completes the proof. \hfill~\qed
\end{proof}

\begin{remark}[Epidemiological Interpretation] 
\label{rem:BioInt1} Understanding the stability properties of disease-free and endemic equilibria is a fundamental issue in mathematical epidemiology.  Much of the literature on this topic investigates the existence of threshold parameters which can be used to identify potential epidemic outbreaks.  The previous result suggests that the joint Lyapunov exponent may be used as a threshold parameter for time-varying epidemiological systems described by switched SIS models.  Recall that the system matrices $B_j$ and $D_j$ are determined by the contact rates between different patches or subgroups in the population as well as by the birth and death rates and the rates at which infectives are cured.  Theorem \ref{t:stabimpliesasstab} applies to situations where these parameters are allowed to vary in time.  Subject to the irreducibility assumption, it establishes that the disease-free equilibrium is globally asymptotically stable for every measurable switching signal provided the linearisation is stable.  Thus writing $\rho$ for the JLE, we have a threshold-type condition with the critical value being $\rho = 0$.  

An advantage of this result is that no precise knowledge of the temporal pattern underlying the parameter variation is required.  In this sense, it allows us to conclude that a disease will eventually die out irrespective of how the key epidemiological parameters vary in time.  Results of this nature are potentially useful for public health authorities as it may not in general be possible to know precisely when and how the contact patterns between different subgroups of a population will change.  Of course, given that the conclusions are so strong and that the result essentially covers a ``worst-case'', it is possible that the conditions provided by Theorem \ref{t:stabimpliesasstab} are practically conservative; however, given that we are dealing with questions of public health, this may not be so serious an objection.

\end{remark}

\section{Persistence and Periodic Orbits}
\label{sec:Per} In the previous section, we considered the
asymptotic stability of the Disease Free Equilibrium (DFE) for a switched SIS
model, all of whose constituent systems possess a globally
asymptotically stable DFE.  We now describe conditions in which all
constituent systems have a GAS DFE but for which there exist
switching laws that give rise to endemic behaviour.  Throughout the section, $D_1, \ldots, D_m$ are diagonal matrices with positive
entries along the main diagonal and $B_1, \ldots , B_m$ are
\textit{irreducible} nonnegative matrices.  We assume throughout
that for $1 \leq j \leq m$, $\mu(-D_j + B_j) < 0$ so that the DFE is
globally asymptotically stable for each constituent system of
(\ref{eq:Fallsw}).  We first show that if there is a matrix $R$ in
the convex hull $\conv\{-D_1 + B_1, \ldots , -D_m + B_m\}$ with $\mu(R) >
0$, then there exists a switching signal $\sigma \in \mathcal{S}$
for which the associated system (\ref{eq:Fallsw}) is strongly
persistent.  As $\Sigma_n$ is forward invariant under (\ref{eq:Fallsw}) for all switching signals, in this section we take $\Sigma_n$ to be the state space.

We first recall a result concerning averaging for time-varying differential equations, which will prove very useful for the analysis in this section.  Consider the time-varying differential equation
\begin{equation}
\label{eq:time-var}
\dot{x}(t) = f(x, t)\,, \quad \; x(0) = x^0\,,
\end{equation}
where $f:\Sigma_n \times \RR_+ \rightarrow \RR^n$ is $K$-Lipschitz in $x$ and measurable in $t$.  Furthermore, we assume that $f$ is bounded by $r$ and periodic with period $T$, so that $f(x, t+T) = f(x, t)$ for all $t \in \RR_+$, $x \in \Sigma_n$.  Define the averaged system
\begin{equation}
\label{eq:avera} \dot{x} = f_0(x)\,, \quad\; x(0) = x^0\,,\quad \text{ where }
f_0(x) = \frac{1}{T} \int_0^T f(x, s) ds.\end{equation}
Combining Theorem 4.1 of \cite{Art07} with Remark 7.1 of the same paper establishes the following fact.
\begin{theorem}
\label{thm:Aver1} Let $x(t, x^0)$, $y(t, x^0)$ denote the solutions of the systems (\ref{eq:time-var}), (\ref{eq:avera}) respectively.  Then for $t \in [0, 1]$,
\[\|x(t, x^0) - y(t, x^0)\|_{\infty} < T(re^{2K}(2+K)).\]
\end{theorem}

In applying Theorem \ref{thm:Aver1} to the switched system (\ref{eq:Fallsw}), we shall take $f(x, t)$ to be of the form
\[f(x, t) = f_{\sigma(t)}(x)\] for some $\sigma \in \mathcal{S}$.  As each $f_j$ in the definition of (\ref{eq:Fallsw}) is $C^1$ on $\RR^n$, it follows that for any $\sigma$, $f_{\sigma(t)}(x)$ is $K$-Lipschitz in $x$ and bounded for $x$ in the state space $\Sigma_n$.

\begin{proposition}
\label{prop:persis} Consider the switched SIS model
(\ref{eq:Fallsw}).  Let $\mu(-D_j + B_j) < 0$ for $1 \leq j \leq m$
and assume that there exists some $R \in \conv\{-D_1+B_1,
\ldots, -D_m + B_m\}$ with $\mu(R) > 0$.  Then there exists $\sigma
\in \mathcal{S}$ such that for all $x^0 > 0$, $1 \leq i \leq n$
\[\liminf_{t \rightarrow \infty} x_i(t, x^0, \sigma) > 0.\]
\end{proposition}
\begin{proof}
By assumption, $R$ is in $\conv\{-D_1+B_1,
\ldots, -D_m + B_m\}$ and hence can be written as
\begin{equation}
\label{eq:R} R = \kappa_1 (-D_1 + B_1) +\ldots + \kappa_m (-D_m + B_m),
\end{equation}
where $\kappa_j \geq 0$ and $\sum_{j=1}^m \kappa_j = 1$.  Define
$\hat{D} = \kappa_1D_1 + \cdots + \kappa_m D_m$, $\hat{B} =
\kappa_1B_1 + \cdots + \kappa_m B_m$ and consider the autonomous SIS
model given by
\begin{equation}
\label{eq:AvSIS} \dot{x}(t) = (-\hat{D} + \hat{B}) x -
\textrm{diag}(x) \hat{B} x.
\end{equation}
For any $T > 0$, we can define a periodic switching signal $\sigma \in
\mathcal{S}$ as follows.
\begin{eqnarray}
\label{eq:persig}
\sigma(t) &=& 1 \quad \mbox{ for } 0 \leq t < \kappa_1 T \\
\nonumber
\sigma(t) &=& j \quad \mbox{ for } \sum_{k=1}^{j-1} \kappa_k T \leq t < \sum_{k=1}^{j} \kappa_k T, \;\; 2 \leq j \leq m.
\end{eqnarray}
Finally, $\sigma(t + T) = \sigma(t)$ for all
$t \geq 0$.

For ease of notation, we shall write
\begin{equation}\label{eq:fsigma}
f_{\sigma(t)}(x) = (-D_{\sigma(t)} + B_{\sigma(t)})x -
\textrm{diag}(x) B_{\sigma(t)} x,
\end{equation}
and similarly we write
\begin{equation}\label{eq:fhat}
\hat{f}(x) = (-\hat{D} + \hat{B}) x - \textrm{diag}(x) \hat{B} x.
\end{equation}
From Theorem~\ref{thm:Fall2}, there is an endemic equilibrium
$\hat{x}$ for
\begin{equation}\label{eq:sysfhat}
\dot{x}(t) = \hat{f}(x(t))
\end{equation}
which is asymptotically stable with region of attraction
$\Sigma_n \setminus \{0\}$.  Let $\phi(\cdot,
x^0)$ denote the solution of (\ref{eq:sysfhat}) with initial
condition $x^0$.

As $\mu(-\hat{D} + \hat{B}) > 0$, there exists some $v \gg 0$ with
$(-\hat{D} + \hat{B}) v \gg 0$.  As the second term in (\ref{eq:fhat}) is
quadratic in $x$, by suitably scaling $v$ (by a constant less than 1), we
can ensure that $\hat{f}(v) \gg 0$.  It follows from Lemma
\ref{lem:nonincreas} that $\phi(t, v)$ is increasing.  We show that this
implies that the solution $x(\cdot, v, \sigma)$ 
satisfies $\liminf_{t \rightarrow \infty} x_i(t, v, \sigma) > 0$ for $1
\leq i \leq n$. Note first that for all $x$
\[\hat{f}(x) = \frac{1}{T} \int_0^T f_{\sigma(s)}(x) ds\]
It follows from Theorem~\ref{thm:Aver1} that for any $\varepsilon > 0$, we can choose $T$
to ensure that
\[ \|x(s, x^0, \sigma) - \phi(s, x^0)\|_{\infty} < \varepsilon\]
for all $s \in [0,1]$, $x^0 \in \Sigma_n$.  In particular, as $\phi(s, v)$ is increasing, we can guarantee by appropriate choice of $\varepsilon$ that $x(s, v, \sigma) \gg 0$ for all $s \in [0, 1]$ and that
$x(1, v, \sigma) \gg v$.  We can without loss of generality assume that $T = \frac{1}{N_0}$ for some integer $N_0$.  Let
$\delta = \min_i \textrm{inf}\{ x_i(s, v, \sigma) : 0 \leq s \leq 1\}$.
Then $\delta > 0$.

By our choice of periodic $\sigma$ (with period $T = 1/N_0$), $x(1 + t, v, \sigma) = x(t,
x(1, v, \sigma), \sigma)$ for $0 \leq t \leq 1$.  As
$f_{\sigma(t)}(\cdot)$ is cooperative for any fixed $t$, it follows
that
\[x(1 + t, v, \sigma) \gg x(t, v, \sigma)\]
for $0 \leq t \leq 1$ and hence $x_i(1+t, v, \sigma) \geq \delta$ for $0 \leq t \leq 1$, $1 \leq i \leq n$.  Iterating, we see that
\[x_i(t, v, \sigma) \geq \delta\] 
for all $t \geq 0$, $1 \leq i \leq n$.
As $f_{\sigma(t)}$ is cooperative for any $t$, it immediately follows that for all $x^0 \geq v$ we have
$x(t, x^0,\sigma) \geq \delta$
for all $t \geq 0$.
Suppose we are given some $x^0$ with $0 \ll x^0 \ll v$. We can
choose some $\lambda < 1$ such that $\lambda v \ll x^0$.
From the form of each $f_j$, it is easy to see that for all $t \geq
0$, $\lambda < 1$ and $x \geq 0$,
\[f_\sigma(t)(\lambda x) \geq \lambda f_{\sigma(t)}(x).\]
Now define $z(t) = \lambda x(t, v, \sigma)$.  Clearly
\[\dot{z}(t) = \lambda f_{\sigma(t)}(x(t, v, \sigma)) \leq f_{\sigma(t)}(z(t)).\]
It now follows from results on differential inequalities (see for instance Theorem A.19 of \cite{SmiThi}) that $z(t) \leq x(t, \lambda v, \sigma)$ for $t \geq 0$.  Hence,
\[x(t, \lambda v, \sigma) \geq \lambda x(t, v, \sigma)\]
for $t \geq 0$.  This immediately implies that
$x_i(t, x^0, \sigma) \geq \lambda \delta$ for all $t \geq 0$, $1 \leq i \leq n$.

Thus far, we have shown that for all initial conditions $x^0 \gg 0$,
\[\liminf_{t \rightarrow \infty} x_i(t, x^0, \sigma) > 0\]
for $1 \leq i \leq n$. To finish the proof, note that for any $x^0 >
0$, as each $B_i$ is irreducible, it follows that $x(t, x^0, \sigma)
\gg 0$ for all $t > 0$ (see Theorem 4.1.1 of \cite{Smi95}).  It is now simple to adapt the above argument to show that in this case also $\liminf_{t \rightarrow \infty} x_i(t, x^0, \sigma) > 0$ for $1 \leq i \leq n$.  This completes the proof.\hfill~\qed
\end{proof}

\begin{remark}[Epidemiological Interpretation] 
\label{rem:BioInt2}
Proposition \ref{prop:persis} illustrates how switched SIS models can exhibit more complicated behaviour than autonomous models.  In particular, it shows that even when each constituent system has a globally asymptotically stable disease free equilibrium, it is possible for the disease to persist in each population subgroup.  When the conditions of the proposition are satisfied, there exists a periodic switching rule $\sigma$ such that $\sup_{T > 0} (inf_{t \geq T} x_i(t)) > 0$.  Epidemiologically, this has the following interpretation.  For this switching rule, not only is the disease free equilibrium not asymptotically stable; there is some time $T_0$ such that a positive fraction of the population of every subgroup is infected at all times after $T_0$.  This reflects classical endemic behaviour where the disease becomes a ``fact of life'' in the population. 
\end{remark} 

The next result shows that under the same assumptions as in
Proposition~\ref{prop:persis}, there exists $\sigma \in \mathcal{S}$
for which (\ref{eq:Fallsw}) admits a periodic orbit in
 $\intt{\mathbb{R}^n_+}$.  To show this result, we use some facts from the degree theory of continuous mappings.  For background on this topic, see \cite{Fac, Lloyd}.  In particular, the facts we use are drawn from Theorem 2.1.2 and Proposition 2.1.3 of \cite{Fac}.  For convenience, we now recall the main points required in our analysis.

Let $\Omega$ be an open region in $\RR^n$ and let $F$ be a continuous function from $\overline{\Omega}$ into $\RR^n$.  If $F(x) \neq 0$ for all $x \in \textrm{bd}(\Omega)$ and $F(\hat x) = 0$ for some $\hat x \in \Omega$, then the degree $\textrm{deg}(F, \Omega, 0) \neq 0$.  Conversely, if $F(x) \neq 0$ for all $x \in \textrm{bd}(\Omega)$ and $\textrm{deg}(F, \Omega, 0) \neq 0$ then there is some $\hat x$ in $\Omega$ with $F(\hat x) = 0$.  Furthermore, if $G$ is another continuous mapping from $\overline{\Omega}$ into $\RR^n$ and
\[\max_{x \in \overline\Omega} \|F(x) - G(x) \|_{\infty} < \inf_{x \in \textrm{bd}(\Omega)} \|F(x)\|_{\infty}\]
then $\textrm{deg}(F, \Omega, 0) = \textrm{deg}(G, \Omega, 0)$.

\begin{theorem}
\label{thm:period}Consider the switched SIS model (\ref{eq:Fallsw}).
Let $\mu(-D_j + B_j) < 0$ for $1 \leq j \leq m$ and assume that there
exists some $R \in \conv\{-D_1+B_1, \ldots, -D_m + B_m\}$ with
$\mu(R) > 0$.  Then there exists $\sigma \in \mathcal{S}$ such that
(\ref{eq:Fallsw}) admits a periodic orbit
\[x(t + 1, x^0, \sigma) = x(t, x^0, \sigma) \quad \forall t \geq 0.\]
\end{theorem}
\begin{proof} As in the proof of Proposition~\ref{prop:persis},
let $R$ be given by (\ref{eq:R}).  Also for $T > 0$, let $\sigma$ in
$\mathcal{S}$ be the switching signal defined by (\ref{eq:persig}).  Let $f_{\sigma(t)}$ and $\hat{f}$ be
given by (\ref{eq:fsigma}) and (\ref{eq:fhat}) respectively.  Finally, we use
$\phi(\cdot, x^0)$ and $x(\cdot, x^0, \sigma)$ to denote the solutions of (\ref{eq:sysfhat}) and (\ref{eq:fsigma}) respectively.

Consider the two continuous mappings defined for $x^0 \in \Sigma_n$ by
\[S_1(x^0) := \int_0^1 \hat{f}(\phi(s, x^0)) ds\,,\quad
S_2(x^0) := \int_0^1 f_{\sigma(s)}(x(s, x^0, \sigma)) ds\,.\]

It follows from Theorem~\ref{thm:Fall2} that there exists
$\hat{x}$ in $\intt(\Sigma_n)$ with $\hat{f}(\hat{x}) = 0$.
Moreover, $\hat{x}$ is an asymptotically stable equilibrium of
(\ref{eq:sysfhat}) with region of attraction $\Sigma_n
\setminus \{0\}$.  This implies that $S_1(\hat{x}) =
0$ and, moreover that $S_1(x) \neq 0$ for $x \in \intt(\Sigma_n) \backslash \{\hat x\}$.  In particular, we can choose some bounded open neighbourhood $\Omega
\subset \intt{\Sigma_n}$ of $\hat{x}$ such that $S_1(z) \neq 0$
for all $z \in \textrm{bd}(\Omega)$.  Let
\[\varepsilon = \textrm{min}\{\|S_1(z)\|_{\infty} : z \in \textrm{bd}(\Omega)\}.\]
As $S_1(z) - S_2(z) = \phi(1, z) - x(1, z, \sigma)$, it follows from Theorem~\ref{thm:Aver1} that we can choose $T = (1/N_0)$ for some positive integer $N_0$ 
such that
\[\textrm{max}_{z \in \overline{\Omega}} \| S_1(z) - S_2(z)\|_{\infty} < \varepsilon.\]
It now follows from the remarks on degree theory given before this theorem that
$\textrm{deg}(S_1, \Omega, 0) = \textrm{deg}(S_2, \Omega, 0)$ and
hence that there exists some $x^1 \in \Omega$ with $S_2(x^1) = 0$.
It follows immediately that \[x(1, x^1, \sigma) = x(0, x^1,
\sigma)\] and as $\sigma$ is $T$-periodic with $T = 1/N_0$, we have that $x(t + 1,
x^1, \sigma) = x(t, x^1, \sigma)$ for all $t \geq 0$.
As $\mu(-D_j + B_j) < 0$ for $1 \leq j \leq m$, (\ref{eq:Fallsw}) possesses no equilibrium in
$\intt(\Sigma_n)$.  Hence, it follows that $x^1$ gives rise to the
claimed periodic orbit.\hfill~\qed
\end{proof}

\begin{remark}
\label{rem:BioInt3} Similarly to Proposition \ref{prop:persis}, Theorem \ref{thm:period} illustrates a form of endemic behaviour that can emerge in the switched model under certain conditions.  When a matrix $R$ with $\mu(R) > 0$ exists there is some switching scheme (pattern of time-variation in the system parameters) for which a periodic solution exists.  Epidemiologically, such a solution amounts to a repetitive seasonal pattern in the number of infectives in the population; the numbers may decrease for some time but eventually return to their initial values and the pattern then simply repeats itself for all time.   As we noted in Remark \ref{rem:BioInt3}, this reflects the disease persisting and becoming a feature of life for the population.
\end{remark}
\section{Markovian Switching}
\label{sec:markov}
We now examine the stability of the switched SIS model
(\ref{eq:Fallsw}) in which the switching parameter $\sigma : \mathbb{R}_+
\rightarrow \{1,...,m\}$ is given as the realization of a right-continuous
random process, in particular, a piecewise deterministic process.  One way
of specifying such a process is by prescribing the rates $\pi_{ij}(t)$
describing the evolution of the probabilities to switch from $i$ to $j$,
see \cite{Norris} for details.

Let $(\Omega, { \cal B}, \Prob)$ be a probability space, we suppose that
the switching signals $\sigma$ are realizations of a Markov process
described by the following transition probabilities
\begin{equation}
\Prob \{\sigma(t+\Delta)=j|\sigma(t)=i\} =  \left\{
\begin{array}{lr}
\pi_{ij}(t)\Delta+o(\Delta) &\mbox{if } i\neq j, \\
1+\pi_{ii}(t)\Delta +o(\Delta) &\mbox{else,}
\end{array} \right.
\end{equation}
where $\Delta >0$ and $\dsp\lim_{\Delta \rightarrow 0}
o(\Delta)/\Delta=0$. The matrix $\Pi(t)=[\pi_{ij}(t)]$ is the matrix
of the transition probability rates and its components  satisfy
$\pi_{ij} \geq 0$ for $i \neq j$, and $\pi_{ii} = -\dsp\sum_{j\neq
i}^{m}\pi_{ij} $.

In order to be precise about the stability of the random system
(\ref{eq:Fallsw}), we need to define some stability concepts.  We denote
by $\Expect(x(t))$ the expectation of $x(t)$.
\begin{definition}\label{def-stab}
  The  switched SIS model
  (\ref{eq:Fallsw}) under a random switching $\sigma$ is  said to be
 \begin{itemize}
 \item[(i)]
 mean stable  if   $\dsp\lim_{t\rightarrow +\infty} \Expect(x(t))=0$,\; $\forall x(0)
 \in \mathbb{R}^n_+$.
 \item[(ii)]
 mean-square stable if $\dsp\lim_{t\rightarrow +\infty}
 \Expect(x(t)^Tx(t))=0$,\; $\forall x(0) \in \mathbb{R}^n_+$.
 \item[(iii)]
 $ L_1$-stable if    $\|x\|_1 :=\int_0^{+\infty}
 \Expect(\dsp\sum_{i=1}^n x_i(t))dt < +\infty$,\; $\forall x(0) \in
 \mathbb{R}^n_+$.
 \item[(iv)]
 $ L_2$-stable if    $\| x\|_2 :=\int_0^{+\infty}
 \Expect(x^T(t)x(t))dt <  +\infty$,\; $\forall x(0) \in
 \mathbb{R}^n_+$.
 \end{itemize}
\end{definition}

Now, define $\delta_i(\cdot)$ as the random processes given by the indicator
function of the Markovian switching process $\sigma$, i.e.
\[
\delta_i(t)=1\;\; \mbox{ if } \sigma(t)=i,\quad \quad
\delta_i(t)=0\;\; \mbox{ if } \sigma(t)\neq i\,.
\]
The following lemma provides a  connection between the
switched SIS model (\ref{eq:Fallsw}) and a special  deterministic
system.
\begin{lemma}\label{lem-connection}
The state $x(t)$ of the stochastic switched
system~(\ref{eq:Fallsw}) satisfies the
differential  equation
\begin{equation}\label{Mom-syst}
\dot \xi_i(t)=(-D_i+B_i)\xi_i-\Expect[\delta_i(t)\diag(x(t))B_i)x(t)]
+\sum_{j=1}^m\pi_{ji}(t)\xi_j(t), \;\; i=1,\ldots,m,
\end{equation}
 where $ \xi_i(t)=\Expect(\delta_i(t)x(t))$.
\end{lemma}
\begin{proof}
Follows using the generalized It\^o formula for Markov jumps,
see e.g. \cite{Bjo80}.\hfill~\qed
\end{proof}

 Note that  since $\dsp\sum_{i=1}^m
\delta_i(t)=1,$  the expectation of any trajectory $x(t)$ of the
switched SIS model (\ref{eq:Fallsw}) is given by
\begin{equation}\label{equa-mean}
\Expect(x(t))=\dsp\sum_{i=1}^m\xi_i(t).
\end{equation}

\begin{remark}\label{rem-positive-mean}
We can deduce from Lemma \ref{lem-connection} that the switched SIS
model (\ref{eq:Fallsw}) is positive if and
 only if the  system~(\ref{Mom-syst}) is positive. This  can be  easily  shown from
 the fact that $B_i\geq 0$, for $i=1\ldots,m$
  and $\pi_{ij} \geq 0$ for $i \neq j$.
\end{remark}

In the sequel,  we shall investigate   the stability of the
deterministic system (\ref{Mom-syst}) in the state variable
$\xi(t)=[\xi_1(t) \ldots
 \xi_m(t)]^T$. This will allow us to derive  stability conditions for
 the associated switched SIS model (\ref{eq:Fallsw}).

\begin{remark}\label{rem-stability}
An immediate consequence of the   relation (\ref{equa-mean}) is that
the switched SIS model (\ref{eq:Fallsw}) is mean stable if and only
if the deterministic system~(\ref{Mom-syst}) is globally
asymptotically stable. Indeed,  since $x(t)\geq 0$, we also have
$\xi_i(t)=\Expect(\delta_i(t) x(t))\geq 0$. Hence,
$\Expect(x(t))=\dsp\sum_{i=1}^m \xi_i(t)$  goes to zero if and only
if all  $\xi_i(t)$  go to zero. Also, it can easily be seen that the
switched SIS model (\ref{eq:Fallsw}) is $L_1$-stable if and only if
the deterministic system~(\ref{Mom-syst}) is  $L_1$-stable.
\end{remark}

For ease of notation, we shall use the following matrices
\begin{equation}\label{A-B-matrices}
 {\cal A}_\Pi:=\left[\begin{array}{ccc} -D_1+B_1& & 0\\
&  \ddots &  \\
0& &-D_m+B_m
\end{array}\right]+\Pi\otimes I \;
\end{equation}

\[ {\cal
B}(x(t)):=\left[\begin{array}{c}\Expect[\delta_1(t)\diag(x(t))B_1x(t)]\\
\vdots\\
\Expect[\delta_m(t)\diag(x(t))B_mx(t)]
\end{array}\right]\,,
\text{ where } \quad
\Pi\otimes I=\left[\begin{array}{ccc} \pi_{11}I&\ldots  & \pi_{1m}I\\
\vdots&  \ddots&  \vdots\\
 \pi_{m1}I & \ldots
& \pi_{mm}I
\end{array}\right]
\] 
is the Kronecker product of $\Pi$ and the
identity. Thus, the system~(\ref{Mom-syst}) can be expressed in
compact form as
\begin{equation}\label{Mom-syst-compact-form}
 \dot \xi(t)={\cal A}_\Pi \xi(t)-{\cal B}(x(t)), \mbox{ where } \xi(t)=[\xi_1(t) \ldots \xi_m(t)]^T.
\end{equation}

We now derive $L_1$  stability conditions for the switched SIS
model (\ref{eq:Fallsw}).

\begin{theorem}\label{stab-analysis1}
    Assume that $\Pi(\cdot)$ is bounded, that is, there exists a constant
    Metzler matrix $\bar \Pi$ such that $\Pi(t) \leq \bar \Pi$ for
    all $t \geq 0$.  Then, the switched SIS model (\ref{eq:Fallsw}) is
    $L_1$-stable if the matrix ${\cal A}_{\bar \Pi}$ is Hurwitz.
\end{theorem}
\begin{proof}
Taking into account  Remark \ref{rem-stability}, it suffices to
prove that the system~(\ref{Mom-syst-compact-form}) is $L_1$-stable.
The proof uses the well-known  condition that a Metzler
matrix ${\cal A}_{\bar \Pi} $ is Hurwitz if and only if  ${\cal
A}^{-1}_{\bar \Pi}\leq 0$.
Now, since $\dot \xi(t)= {\cal A}_{\Pi}\xi(t) -{\cal B}(x(t))$, we obtain
\[\xi(t)-\xi(0)=\int_0^t {\cal A}_\Pi\xi(s) -{\cal B}(x(s))ds.\]
 As  ${\cal B}(x(t))\geq 0$ and ${\cal A}_{\Pi} \leq {\cal A}_{\bar \Pi}$ we have
$\xi(t)-\xi(0)\leq\int_0^t {\cal A}_{\bar \Pi}\ \xi(s)ds$.
The assumption that ${\cal A}^{-1}_{\bar \Pi}\leq 0$ leads to
\[-{\cal A}^{-1}_{\bar \Pi}\xi(t)+{\cal A}^{-1}_{\bar \Pi}\xi(0)\leq  -\int_0^t\xi(s)ds.\]
Since  $-{\cal A}^{-1}_{\bar \Pi} \xi(t)\geq0$ we obtain
\[\int_0^t\xi(s)ds \leq  -{\cal A}^{-1}_{\bar \Pi}\xi(0),\; \forall t>0.\]

The above inequality shows that the integral $\int_0^t\xi(s)ds$ is bounded
and as it is nondecreasing, so that the integral $\int_0^{+\infty}
\xi(t)dt$ exists and the proof is complete.\hfill~\qed
\end{proof}

Since $L_1$ stability implies $L_2$ stability, then we can deduce
from  Theorem~\ref{stab-analysis1} that  if  $\mu({\cal A}_{\bar
\Pi})<0$, then the switched SIS model (\ref{eq:Fallsw}) is
$L_2$-stable. Indeed,  it is also mean stable and mean-square
stable.

\begin{remark}\label{rem-positive-mean2}
We emphasize that we  cannot apply Theorem~\ref{thm:Fall1}
directly to system~(\ref{Mom-syst}) because it involves a different
structure. Also, Theorem~\ref{stab-analysis1} establishes an
$L_1$ stability condition that also holds for the deterministic
system (\ref{eq:Fall}) (when $\Pi=0$).
\end{remark}

\section{Stabilization of the Disease Free Equilibrium by Switching}
\label{sec:Stabilisation} In this section, we assume that each constituent system of (\ref{eq:Fallsw})
has an endemic equilibrium, which is asymptotically stable with
region of attraction $\Sigma_n \setminus \{0\}$.  We shall
show how results from the literature on stabilization of switched
linear systems can be applied to define switching
laws that asymptotically stabilise the Disease Free Equilibrium (DFE).




\begin{theorem}\label{thm:stabilis}
Consider the switched SIS model (\ref{eq:Fallsw}).  Assume that
$\mu(-D_j + B_j) > 0$
 for $1 \leq j \leq m$.  Suppose that there exists some
  $R \in \conv\{-D_1 + B_1, \ldots , -D_m + B_m \}$ for which
$\mu(R) < 0$.  Then, there exists some $T>0$ and a
periodic switching law $\sigma \in \mathcal{S}$ with period $T$  such
that the DFE of (\ref{eq:Fallsw}) is GAS.
\end{theorem}
\begin{proof} Let $R = \dsp\sum_{j=1}^m \kappa_i (-D_j + B_j)$
satisfy $\mu(R) < 0$ where $\kappa_j\geq 0$ for all $j$ and
$\dsp\sum_{j=1}^m \kappa_j = 1$.
For $T > 0$, consider the periodic switching signal $\sigma \in \mathcal{S}$ given by (\ref{eq:persig}).
We claim that it is possible to choose $T$ such that the DFE of (\ref{eq:Fallsw}) is GAS.
To show this, we consider the associated switched linear system
(\ref{eq:Falllinsw}); we denote the solution of (\ref{eq:Falllinsw}) by $\tilde \psi(t, x^0)$.  It is well known that the existence of a
Hurwitz convex combination of the system matrices $-D_j + B_j$ is
sufficient for the existence of a stabilizing switching law (usually state-dependent) for
such systems.  (See for example \cite{Lib03, Sho07}).  We include
the proof here in the interests of completeness.

The averaged system (\ref{eq:avera}) corresponding to (\ref{eq:Falllinsw}) is given by
\begin{equation}
\label{eq:avlin} \dot{x} = R x.
\end{equation}
As $\mu(R) < 0$, we can choose some $v \gg 0$ with $v_i > 1$ for all $i$ and $R v \ll 0$.  The solution $z(t, v)$ of (\ref{eq:avlin}) is decreasing by Lemma \ref{lem:nonincreas} and using Theorem~\ref{thm:Aver1} we can ensure that $\tilde \psi (T, v) \ll v$ by choosing $T$ sufficiently small.  Thus, for such a $T$, we can find some $\alpha$ with $0 < \alpha < 1$ such that
\begin{equation}
\label{eq:alpha}
\tilde \psi (T, v) \leq \alpha v.
\end{equation}
From the construction of $\sigma$ and the linearity of (\ref{eq:Falllinsw}), it follows that for $0 \leq t \leq T$,
\[\tilde\psi(T + t, v) = \tilde\psi(t, \tilde\psi(T, v)) \leq \alpha \tilde\psi(t, v).\]
Iterating the previous identity, we see readily that $\tilde\psi(t, v) \rightarrow 0$ as $t \rightarrow \infty$.  As $v_i \geq 1$ for all $i$, it immediately follows from the monotonicity of the system (\ref{eq:Falllinsw}) that $\tilde\psi(t, x^0) \rightarrow 0$ as $t \rightarrow \infty$ for any $x^0 \geq 0$ with $\|x^0\|_{\infty} \leq 1$.  The result now follows from a simple application of Lemma \ref{lem:comparison}.\hfill~\qed
\end{proof}

\begin{remark}[Epidemiological Interpretation]
\label{rem:BioInt4}
The epidemiological significance of the previous result relates to the development of control strategies for epidemic outbreaks based on switching policies.  In essence, if the conditions of Theorem \ref{thm:stabilis} are satisfied, then it is possible to asymptotically eradicate the disease by choosing a suitable switching signal.  This means that even when each constituent model has an endemic equilibrium, we can drive the system to the disease free state by varying the contact patterns between population subgroups in a suitable manner, for instance. 
\end{remark}

The stabilizing switching strategy described in the above result is
a time-dependent switching signal in the set $\mathcal{S}$.  It is
also possible to prove the existence of state-dependent switching
strategies following arguments similar to those in
\cite[Chapter~3]{Lib03}.

\section{Conclusion and Future Work}\label{sec:conclusions}
Building on the work of \cite{FIST07}, we have described several results
concerning compartmental SIS models with parameters subject to switching.
In particular, Theorem~\ref{t:stabimpliesasstab} generalizes a main
result of \cite{FIST07}, and shows that stability combined with an
appropriate notion of irreducibility for the linearisation is sufficient
for asymptotic stability of the disease free equilibrium of the switched
nonlinear system describing the epidemic dynamics.  As highlighted in the
text, this generalises the result in \cite{FIST07} even for the
time-invariant case.  Our work also indicates how endemic behaviour can
emerge for systems constructed by switching between models, each of which
has a globally asymptotically stable disease free equilibrium.
Specifically, we have described conditions for the disease to be
persistent and for the existence of periodic endemic orbits.  Results for
systems subject to Markovian switching have also been presented.

There are several natural questions arising from the work described here.
We briefly highlight some of these.  While we have
provided conditions for the existence of periodic orbits in
Theorem~\ref{thm:period}, it is natural to ask what stability properties
this orbit possesses.  Is it guaranteed to be locally or globally
attractive?  If not, then what additional requirements will render it so? 
Another question is whether or not the relaxed sufficient condition for
global asymptotic stability of the DFE given in Corollary
\ref{c:fall-complete} is also necessary.  Finally, it would be very
interesting to investigate whether the work of the current paper can be
extended to switched SIR epidemic models.  
\section*{Acknowledgement}
This work was supported by Science Foundation Ireland award
08/RFP/ENE1417 and by the Irish Higher Education Authority PRTLI 4
Network Mathematics grant.

The authors would like to thank Sebastian Pr{\"o}ll for useful discussions
and careful reading of the manuscript.


\begin{thebibliography}{10}

\bibitem{ABMW12}
{\sc M.~{Ait Rami}, V.~S. Bokharaie, O.~Mason, and F.~Wirth}, {\em Extremal
  norms for positive linear inclusions}, in Proc. 20th Int. Symposium on
  Mathematical Theory of Networks and Systems, MTNS 2012, Melbourne, Australia,
  2012.

\bibitem{Art07}
{\sc Z.~Artstein}, {\em Averaging of time-varying differential equations
  revisited}, Journal of Differential Equations, 243 (2007), pp.~146--167.

\bibitem{BacaKhal12}
{\sc N.~Baca{\"e}r and M.~Khaladi}, {\em On the basic reproduction number in a
  random environment}, J. Math. Biol., 67 (2012), pp.~1729--1739.

\bibitem{Bai57}
{\sc N.~T.~J. Bailey}, {\em The Mathematical Theory of Epidemics}, Griffin,
  London, 1957.

\bibitem{bauer1961absolute}
{\sc F.~Bauer, J.~Stoer, and C.~Witzgall}, {\em Absolute and monotonic norms},
  Numerische Mathematik, 3 (1961), pp.~257--264.

\bibitem{BP87}
{\sc A.~Berman and R.~J. Plemmons}, {\em Nonnegative Matrices in the
  Mathematical Sciences}, Classics in Applied Mathematics, SIAM, Philadelphia,
  PA, USA, 1987.

\bibitem{Bjo80}
{\sc T.~Bj{\"o}rk}, {\em Finite dimensional optimal filters for a class of
  {I}to-processes with jumping parameters}, Stochastics, 4 (1980),
  pp.~167--183.

\bibitem{BMW10}
{\sc V.~S. Bokharaie, O.~Mason, and F.~Wirth}, {\em Spread of epidemics in
  time-dependent networks}, Proc. 19th Int. Symposium on Mathematical Theory of
  Networks and Systems, MTNS 2010,  (2010).

\bibitem{Chue02}
{\sc I.~Chueshov}, {\em Monotone Random Systems}, Springer--Verlag, Berlin,
  2002.

\bibitem{ClarLedy98}
{\sc F.~H. Clarke, Y.~S. Ledyaev, R.~J. Stern, and P.~R. Wolenski}, {\em
  Nonsmooth Analysis and Control Theory}, vol.~178 of Graduate Texts in
  Mathematics, Springer-Verlag, New York, 1998.

\bibitem{CoddLevi55}
{\sc E.~A. Coddington and N.~Levinson}, {\em Theory of Ordinary Differential
  Equations}, Mc{G}raw-Hill, 1955.

\bibitem{CY73}
{\sc K.~L. Cooke and J.~A. Yorke}, {\em Some equations modelling growth
  processes and gonorrhea epidemics}, Mathematical Biosciences, 16 (1973),
  pp.~75--101.

\bibitem{Del00}
{\sc P.~{De Leenheer}}, {\em Stabiliteit, regeling en stabilisatie van
  positieve systemen}, PhD thesis, University of Gent, 2000.

\bibitem{DemyRubi95}
{\sc V.~Demyanov and A.~M. Rubinov}, {\em {C}onstructive {N}onsmooth
  {A}nalysis}, Verlag Peter Lang, Frankfurt Berlin, 1995.

\bibitem{Fac}
{\sc F.~Facchinei and J.~Pang}, {\em Finite-Dimensional Variational
  Inequalities and Complementarity Problems}, vol.~1, Springer-Verlag, Berlin,
  2003.

\bibitem{FainMarChig}
{\sc L.~Fainshil, M.~Margaliot, and P.~Chigansky}, {\em On the stability of
  positive linear switched systems under arbitrary switching laws}, IEEE
  Transactions on Automatic Control, 54 (2009), pp.~897--899.

\bibitem{FIST07}
{\sc A.~Fall, A.~Iggidr, G.~Sallet, and J.~Tewa}, {\em Epidemiological models
  and {L}yapunov functions}, Math. Model. Nat. Phenom., 2 (2007), pp.~62--68.

\bibitem{GrayGree12}
{\sc A.~Gray, D.~Greenhalgh, X.~Mao, and J.~Pan}, {\em The {SIS} epidemic model
  with {M}arkovian switching}, Journal of Mathematical Analysis and
  Applications, 394 (2012), pp.~496--516.

\bibitem{GurShoMas}
{\sc L.~Gurvits, R.~Shorten, and O.~Mason}, {\em On the stability of switched
  positive linear systems}, IEEE Transactions on Automatic Control, 52 (2007),
  pp.~1099--1103.

\bibitem{HY84}
{\sc H.~W. Hethcote and J.~A. York}, {\em Gonorrhea Transmission and Control},
  vol.~56 of Lectures Notes in Biomathematics, Springer-Verlag, New York, NY,
  1984.

\bibitem{HornJohn}
{\sc R.~A. Horn and C.~R. Johnson}, {\em Matrix Analysis}, Cambridge University
  Press, New York, NY, USA, 1985.

\bibitem{KaK:60}
{\sc I.~Kats and N.~Krasovskii}, {\em On the stability of systems with random
  parameters}, J. Appl. Math. Mec., 24 (1960), pp.~1225--1246.

\bibitem{Kozy90}
{\sc V.~S. Kozyakin}, {\em Algebraic unsolvability of problem of absolute
  stability of desynchronized systems}, Autom. Rem. Control, 51 (1990),
  pp.~754--759.

\bibitem{KrL:61}
{\sc N.~Krasovskii and E.~Lidskii}, {\em Analytical design of controllers in
  systems with random attributes}, Automation and Remote Control, 22 (1961),
  pp.~1021--1025.

\bibitem{LY76}
{\sc A.~Lajmanovic and J.~Yorke}, {\em A deterministic model for gonorrhea in
  nonhomogeneous population}, Math. Biosci., 28 (1976), pp.~221--236.

\bibitem{Lib03}
{\sc D.~Liberzon}, {\em Switching in Systems and Control}, Birkh{\"a}user,
  Boston, MA, USA, 2003.

\bibitem{PerPatchy2011}
{\sc X.~Liu and X.-Q. Zhao}, {\em A periodic epidemic model with age-structure
  in a patchy environment}, SIAM Journal of Applied Mathematics, 71 (2011),
  pp.~1896--1917.

\bibitem{Lloyd}
{\sc N.~Lloyd}, {\em Degree Theory}, Cambridge University Press, 1978.

\bibitem{Mar:90}
{\sc M.~Mariton}, {\em {Jump Linear Systems in Automatic Control}}, Marcel
  Dekker, New York, NY, 1990.

\bibitem{New03}
{\sc M.~E.~J. Newman}, {\em The structure and function of complex networks},
  SIAM review,  (2003), pp.~167--256.

\bibitem{Norris}
{\sc J.~Norris}, {\em Markov Chains}, Cambridge University Press, Cambridge,
  2008.

\bibitem{Proell13}
{\sc S.~Pr{\"o}ll}, {\em Stability of switched epidemiological models},
  Master's thesis, Institute for Mathematics, University of W{\"u}rzburg,
  W{\"u}rzburg, Germany, 2013.

\bibitem{PerSeasonal2012}
{\sc C.~Rebelo, A.~Margheri, and N.~Baca\"{e}r}, {\em Persistence in seasonally
  forced epidemiological models}, Journal of Mathematical Biology, 64 (2012),
  pp.~933--949.

\bibitem{Sho07}
{\sc R.~Shorten, F.~Wirth, O.~Mason, K.~Wulff, and C.~King}, {\em Stability
  theory for switched and hybrid systems}, SIAM Review, 49 (2007),
  pp.~545--592.

\bibitem{Smi95}
{\sc H.~A. Smith}, {\em Monotone Dynamical Systems. An Introduction to the
  Theory of Competitive and Cooperative Systems}, American Mathematical
  Society, Providence, RI, USA, 1995.

\bibitem{SmiThi}
{\sc H.~L. Smith and H.~R. Thieme}, {\em Dynamical Systems and Population
  Persistence}, American Mathematical Society, Providence, RI, USA, 2011.

\bibitem{DW02}
{\sc P.~van~den Driessche and J.~Watmough}, {\em Reproduction numbers and
  subthreshold endemic equilibria for compartmental models of disease
  transmission}, Math. Biosci.,  (2002), pp.~29--48.

\bibitem{Wirt02}
{\sc F.~Wirth}, {\em The generalized spectral radius and extremal norms}, Lin.
  Alg. Appl., 342 (2002), pp.~17--40.

\bibitem{WHO08}
{\sc {World~Health~Organization~(WHO)}}, {\em The global burden of disease:
  2004 update}.
\newblock
  \url{http://www.who.int/healthinfo/global_burden_disease/2004_report_update/%
en/index.html}, 2008.
\newblock Last Retrieved: 05 December 2011.

\end{thebibliography}
\end{document}